\documentclass{article}

\usepackage{amsmath,amsfonts,amssymb,amsthm}
\usepackage{slashed}
\usepackage{xy}
\usepackage[breaklinks=true]{hyperref}
\usepackage{float}
\usepackage{graphicx}
\usepackage[percent]{overpic}
\usepackage{color}

\definecolor{myurlcolor}{rgb}{0,0,0.7}
\usepackage{hyperref}
\hypersetup{colorlinks,
linkcolor=myurlcolor,
citecolor=myurlcolor,
urlcolor=myurlcolor}

\input xy
\xyoption{all}

\newcommand{\maps}{\colon}    
\newcommand{\R}{{\mathbb R}}  
\newcommand{\C}{{\mathbb C}}  
\renewcommand{\H}{{\mathbb H}}  
\renewcommand{\O}{{\mathbb O}}  
\newcommand{\I}{{\mathbb I}}  
\newcommand{\Z}{{\mathbb Z}}  

\newcommand{\RP}{\mathbb{R}\mathrm{P}} 
\newcommand{\PC}{\mathrm{P}C} 

\renewcommand{\Re}{\mathrm{Re}} 
\renewcommand{\Im}{\mathrm{Im}} 
\newcommand{\tr}{{\mathrm{tr}}} 
\newcommand{\Ann}{\mathrm{Ann}} 

\newcommand{\SO}{{\rm SO}}      
\newcommand{\SU}{{\rm SU}}      
\newcommand{\G}{\mathrm{G}}     

\newcommand{\so}{{\mathfrak{so}}}  
\newcommand{\g}{\mathfrak{g}}      

\newcommand{\Proj}{\mathrm{P}}  

\newcommand{\Ro}{\mathbf{R}} 

\newcommand{\tensor}{\otimes} 

\newcommand{\define}[1]{{\bf \boldmath{#1}}}
\newcommand{\arxiv}[1]{\href{http://arxiv.org/abs/#1}{arXiv:{#1}}}

\newtheorem{thm}{Theorem}    
\newtheorem{cor}[thm]{Corollary}

\newtheorem{prop}[thm]{Proposition}

\theoremstyle{definition}

\newtheorem{defn}[thm]{Definition}

        \newcommand{\be}{\begin{equation}}
        \newcommand{\ee}{\end{equation}}
        \newcommand{\ba}{\begin{eqnarray}}
        \newcommand{\ea}{\end{eqnarray}}
        \newcommand{\ban}{\begin{eqnarray*}}
        \newcommand{\ean}{\end{eqnarray*}}
        \newcommand{\barr}{\begin{array}}
        \newcommand{\earr}{\end{array}}

\title{$\G_2$ and the Rolling Ball}

\author{John C.\ Baez\\[.5em]
{\small Department of Mathematics} \\[-.3em]
{\small  University of California}\\[-.3em]
{\small Riverside, California 92521, USA} \\
\small and \\
{\small Centre for Quantum Technologies}  \\[-.3em]
{\small National University of Singapore} \\[-.3em]
{\small Singapore 117543}  \\
\small  baez@math.ucr.edu 
 \and
{John Huerta} \\[.5em]
{\small CAMGSD} \\[-.3em]
{\small Instituto Superior T\'ecnico} \\[-.3em]
{\small Av.\ Ravisco Pais} \\[-.3em]
{\small 1049-001 Lisboa, Portugal} \\
\small huerta@math.ucr.edu
}

\date{\small August 7, 2012}

\begin{document}

\maketitle

\bigskip \bigskip

\begin{abstract}
\noindent
	Understanding the exceptional Lie groups as the symmetry
	groups of simpler objects is a long-standing program in
	mathematics. Here, we explore one famous realization of the
	smallest exceptional Lie group, $\G_2$.  Its Lie algebra
	$\g_2$ acts locally as the symmetries of a ball rolling on a
	larger ball, but only when the ratio of radii is 1:3.  Using
	the split octonions, we devise a similar, but more global,
	picture of $\G_2$: it acts as the symmetries of a `spinorial
	ball rolling on a projective plane', again when the ratio of
	radii is 1:3.  We explain this ratio in simple terms, use the
	dot product and cross product of split octonions to describe
	the $\G_2$ incidence geometry, and show how a form of
	geometric quantization applied to this geometry lets us
       recover the imaginary split octonions and these operations.
\end{abstract}

\section{Introduction}

When Cartan and Killing classified the simple Lie algebras, they
uncovered five surprises: the exceptional Lie algebras.  The smallest
of these, the Lie algebra of $\G_2$, was soon constructed explicitly
by Cartan and Engel.  However, it was not obvious how to understand
this Lie algebra as arising from the symmetry group of a
naturally occuring mathematical object.  Giving a simple
description of $\G_2$ has been a challenge ever since: though much
progress has been made, the story is not yet finished.

In this paper, we study two famous realizations of the split real form
of $\G_2$, both essentially due to Cartan.  First, this group is the
automorphism group of an 8-dimensional nonassociative algebra: the
split octonions.  Second, it is roughly the group of symmetries of a
ball rolling on a larger fixed ball without slipping or twisting, but
\emph{only when the ratio of radii is 1:3}.

The relationship between these pictures has been discussed before, and
indeed, the history of this problem is so rich that we postpone all
references to the next section, which deals with that history.  We
then explain how each description of $\G_2$ is hidden inside the
other.  On the one hand, a variant of the 1:3 rolling ball system,
best thought of as a `spinor rolling on a projective plane', lives
inside the imaginary split octonions as the space of `light rays':
1-dimensional null subspaces.  On the other hand, we can recover the
imaginary split octonions from this variant of the 1:3 rolling ball
via geometric quantization.

Using a spinorial variant of the rolling ball system may seem odd, but it is
essential if we want to see the hidden $\G_2$ symmetry.  In fact, we must
consider three variants of the rolling ball system.  The first is the ordinary
rolling ball, which has configuration space $S^2 \times \SO(3)$.  This never
has $\G_2$ symmetry.  We thus pass to the double cover, $S^2 \times \SU(2)$,
where such symmetry is possible.  We can view this as the configuration space
of a `rolling spinor': a rolling ball that does not come back to its original
orientation after one full rotation, but only after two.  To connect this
system with the split octonions, it pays to go a step further, and identify
antipodal points of the fixed sphere $S^2$.  This gives $\RP^2 \times \SU(2)$,
which is the configuration space of a spinor rolling on a projective plane.

This last space explains why the 1:3 ratio of radii is so
special. As mentioned, a spinor comes back to its original state only
after two full turns.  On the other hand, a point moving on the
projective plane comes back to its original position after going
halfway around the double cover $S^2$.  Consider a ball rolling
without slipping or twisting on a larger fixed ball.  What must the
ratio of their radii be so that the rolling ball makes two full turns
as it rolls halfway around the fixed one?  Or put another way: what
must the ratio be so that the rolling ball makes four full turns as it
rolls once around the fixed one?  The answer is 1:3.

At first glance this may seem surprising.  Isn't the correct answer
1:4?  

No: a ball of radius 1 turns $R+1$ times as it rolls once around a
fixed ball of radius $R$.  One can check this when $R = 1$ using two
coins of the same kind.  As one rolls all the way about the other
without slipping, it makes two full turns.  Similarly, in our
$365\frac{1}{4}$ day year, the Earth actually turns $366\frac{1}{4}$
times.  This is why the sidereal day, the day as judged by the
position of the stars, is slightly shorter than the ordinary solar
day.  The Earth is not rolling without slipping on some imaginary
sphere.  However, just as with the rolling ball, it makes an `extra
turn' by completing one full revolution around the center of its
orbit.

Of course, this kind of reasoning only takes us so far.  A spinor will
come back to itself after any even number of turns, so the ratios of
radii 1:3, 1:7, 1:11, and so on are all permitted by this
argument---but only the first, 1:3, gives a system with $\G_2$
symmetry.

To understand this a bit better, we should bring the split octonions into the
game.  For any $R > 1$ there is an incidence geometry with points and lines
defined as follows:
\begin{itemize}
\item The points are configurations of a spinorial ball of radius $1$
rolling on a fixed projective plane, with double cover a sphere of radius $R$. 
\item The lines are curves where the spinorial ball rolls
along lines in the projective plane without slipping or twisting.
\end{itemize}
This space of points, $\RP^2 \times \SU(2)$, is the same as the space
of 1-dimensional null subspaces of the imaginary split octonions,
which we call $\PC$. Under this identification, the lines of our
incidence geometry become certain curves in $\PC$.  If and only if $R
= 3$, these curves `straighten out': they are given by projectivizing
certain 2-dimensional null subspaces of the imaginary split octonions.
We prove this in Theorem \ref{thm:nullsubspace}.

Indeed, in the case of the 1:3 ratio, and \emph{only} in this case, we can
find the rolling ball system hiding inside the split octonions.  
A `null subalgebra' of the split octonions is one where the product of any two
elements is zero.  In Theorem \ref{thm:lines_and_2d_null_subalgebras} we show
that when $R = 3$, the above incidence geometry is isomorphic to one where:
\begin{itemize}
\item The points are 1d null subalgebras of the 
imaginary split octonions.
\item The lines are 2d null subalgebras of the 
imaginary split octonions.
\end{itemize}
As a consequence, this geometry is invariant under the automorphism
group of the split octonions: the split real form of $\G_2$.  

This group is also precisely the group that preserves the dot product
and cross product operations on the imaginary split octonions.  These 
are defined by decomposing the octonionic product into real and imaginary
parts:
\[         x y = - x \cdot y + x \times y,  \]
where $x \times y$ is an imaginary split octonion and $x \cdot y$ is a real
multiple of the identity, which we identify with a real number.
One of our main goals here is to give a detailed description of the 
above incidence geometry in terms of these operations.  The 
key idea is that any nonzero imaginary split octonion $x$ with 
$x \cdot x = 0$ spans a 1-dimensional null subalgebra $\langle x \rangle$, 
which is a point in this geometry.  Given two points 
$\langle x \rangle$ and $\langle y \rangle$, we say they are 
`at most $n$ rolls away' if we can get from one to 
the other by moving along a sequence of at most $n$ lines.  Then:
\begin{itemize}
\item $\langle x \rangle$ and $\langle y \rangle$ are at most one roll
away if and only if $x y = 0$, or equivalently, $x \times y = 0 $.
\item $\langle x \rangle$ and $\langle y \rangle$ are at most two rolls
away if and only if $x \cdot y$ = 0.
\item $\langle x \rangle$ and $\langle y \rangle$ are always at most
three rolls away.
\end{itemize}
We define a `null triple' to be an ordered triple of nonzero null imaginary
split octonions $x$, $y$, $z$, pairwise orthogonal, obeying the
condition $(x \times y) \cdot z = \frac{1}{2}$. 
We show that any null triple gives rise to a configuration of 
points and lines like this:
\[
\xy
	(0,10)*{\bullet}="B";
	(2,11.16)*{}="B1";
	(-2,11.16)*{}="B2";
	(8.6,5)*{\bullet}="A";
	(10.6,3.84)*{}="A1";
	(8.6,7)*{}="A2";
	(-8.6,5)*{\bullet}="C";
	(-10.6,3.84)*{}="C1";
	(-8.6,7)*{}="C2";
	(8.6,-5)*{\bullet}="D";
	(10.6,-3.84)*{}="D1";
	(8.6,-7)*{}="D2";
	(-8.6,-5)*{\bullet}="F";
	(-10.6,-3.84)*{}="F1";
	(-8.6,-7)*{}="F2";
	(0,-10)*{\bullet}="E";
	(2,-11.16)*{}="E1";
	(-2,-11.16)*{}="E2";
	"B1";"C1"**\dir{-};
	"B2";"A1"**\dir{-};
	"F1";"E1"**\dir{-};
	"D1";"E2"**\dir{-};
	"A2";"D2"**\dir{-};
	"C2";"F2"**\dir{-};
	"A"+(3,3)*{\langle x \rangle};
	"B"+(0,3)*{\langle x \times y \rangle};
	"C"+(-3,3)*{\langle y \rangle};
	"D"+(4,-4)*{\langle y \times z \rangle};
	"E"+(0,-3)*{\langle z \rangle};
	"F"+(-4,-4)*{\langle z \times x \rangle};
\endxy
\]
In the theory of buildings, this sort of configuration is called an
`apartment' for the group $\G_2$.  Together with $(x \times y) \times
z$, the six vectors shown here form a basis of the imaginary split
octonions, as we show in Theorem \ref{thm:nulltriple}.  Moreover, we
show in Theorem \ref{thm:torsor} that the split real form of $\G_2$
acts freely and transitively on the set of null triples.

We also show that starting from this incidence geometry, we can
recover the split octonions using geometric quantization.  The space
of points forms a projective real variety,
\[    \PC   \cong \RP^2 \times \SU(2) . \]
There is thus a line bundle $L \to \PC$ obtained by restricting the
dual of the canonical line bundle to this variety.  Naively, one might
try to geometrically quantize $\PC$ by forming the space of
holomorphic sections of this line bundle.  However, since $\PC$ is a
\emph{real} projective variety, and $L$ is a \emph{real} line bundle,
the usual theory of geometric quantization does 
not directly apply.  Instead we need a slightly more elaborate
procedure where we take sections of $L \to \PC$ that extend to
holomorphic sections of the complexification $L^\C \to \PC^\C$.  In
Theorem \ref{thm:real_sections} we prove the space of such sections is
the imaginary split octonions.  In Theorem \ref{thm:cross_product},
we conclude by using geometric quantization to reconstruct the cross
product of imaginary split octonions, at least up to a constant factor.

\section{History}

On May 23, 1887, Wilhelm Killing wrote a letter to Friedrich Engel
saying that he had found a 14-dimensional simple Lie algebra
\cite{Agricola}.  This is now called $\g_2$.  By October he had
completed classifying the simple Lie algebras, and in the next three years
he published this work in a series of papers \cite{Killing}.  Besides
the already known classical simple Lie algebras, he claimed to have
found six `exceptional' ones.  In fact he only gave a rigorous
construction of the smallest, $\g_2$.  In his 1894 thesis, \'Elie
Cartan \cite{Cartan:classification} constructed all of them and
noticed that two of them were isomorphic, so that there are really
only five.

But already in 1893, Cartan had published a note \cite{Cartan:note}
describing an open set in $\C^5$ equipped with a 2-dimensional
`distribution'---a smoothly varying field of 2d spaces of tangent
vectors---for which the Lie algebra $\g_2$ appears as the
infinitesimal symmetries.  In the same year, in the same journal,
Engel \cite{Engel:note} noticed the same thing.  As we shall see, this
2-dimensional distribution is closely related to the rolling ball.
The point is that the space of configurations of the rolling ball is
5-dimensional, with a 2-dimensional distibution that describes motions
of the ball where it rolls without slipping or twisting.

Both Cartan \cite{Cartan:five} and Engel \cite{Engel:G2} returned to
this theme in later work.  In particular, Engel discovered in 1900
that a generic antisymmetic trilinear form on $\C^7$ is preserved by a
group isomorphic to the complex form of $\G_2$.  Furthermore, starting
from this 3-form he constructed a nondegenerate symmetric bilinear
form on $\C^7$.  This implies that the complex form of $\G_2$ is
contained in a group isomorphic to $\SO(7,\C)$.  He also noticed that
the vectors $x \in \C^7$ that are null---meaning $x \cdot x =
0$, where we write the bilinear form as a dot product---define a
5-dimensional projective variety on which $\G_2$ acts.

As we shall see, this variety is the complexification of the
configuration space of a rolling spinorial ball on a projective plane.
Futhermore, the space $\C^7$ is best seen as the complexification of
the space of imaginary octonions.  Like the space of imaginary
quaternions (better known as $\R^3$), the 7-dimensional space of
imaginary octonions comes with a dot product and cross product.
Engel's bilinear form on $\C^7$ arises from complexifying the dot
product.  His antisymmetric trilinear form arises from the dot product
together with the cross product via the formula $x \cdot (y \times
z)$.

However, all this was seen only later.  It was only in 1908 that
Cartan mentioned that the automorphism group of the octonions is a
14-dimensional simple Lie group \cite{Cartan:octonions1}.  Six years
later he stated something he probably had known for some time: this
group is the compact real form of $\G_2$ \cite{Cartan:octonions2}.

The octonions had been discovered long before, in fact the day after
Christmas in 1843, by Hamilton's friend John Graves.  Two months
before that, Hamilton had sent Graves a letter describing his dramatic
discovery of the quaternions.  This encouraged Graves to seek an even
larger normed division algebra, and thus the octonions were born.
Hamilton offered to publicize Graves' work, but put it off or forgot
until the young Arthur Cayley rediscovered the octonions in 1845
\cite{Cayley}.  That this obscure algebra lay at the heart of {\em
all} the exceptional Lie algebras became clear only slowly
\cite{Baez:octonions}.  Cartan's realization of its relation to
$\g_2$, and his later work on triality, was the first step.

In 1910, Cartan wrote a paper that studied 2-dimensional distributions
in 5 dimensions \cite{Cartan:five}.  Generically such a distibution is
not integrable: the Lie bracket of two vector fields lying in this
distribution does not again lie in this distribution. However, near a
generic point, it lies in a 3-dimensional distribution.  The Lie
bracket of vector fields lying in this 3-dimensional distibution then
generically give arbitary tangent vectors to the 5-dimensional
manifold.  Such a distribution is called a `$(2,3,5)$
distribution'.  Cartan worked out a complete system of local geometric
invariants for these distributions.  He showed that if all these
invariants vanish, the infinitesimal symmetries of a
$(2,3,5)$ distribution in a neighborhood of a point form the Lie
algebra $\g_2$.

Again this is relevant to the rolling ball.  The space of
configurations of a ball rolling on a surface is 5-dimensional, and it
comes equipped with a $(2,3,5)$ distribution.  The 2-dimensional
distibution describes motions of the ball where it rolls without
twisting or slipping.  The 3-dimensional distribution describes
motions where it can roll and twist, but not slip.  Cartan did
not discuss rolling balls, but he does consider a closely related
example: curves of constant curvature 2 or 1/2 in the unit 3-sphere.

Beginning in the 1950's, Fran\c{c}ois Bruhat and Jacques Tits
developed a very general approach to incidence geometry, eventually
called the theory of `buildings' \cite{Brown,Tits}, which among other
things gives a systematic approach to geometries having simple Lie
groups as symmetries.  In the case of $\G_2$, because the Dynkin
diagram of this group has two dots, the relevant geometry has two
types of figure: points and lines.  Moreover because the Coxeter group
associated to this Dynkin diagram is the symmetry group of a hexagon,
a generic pair of points $a$ and $d$ fits into a configuration like this,
called an `apartment':
\[
\xy
	(0,10)*{\bullet}="B";
	(2,11.16)*{}="B1";
	(-2,11.16)*{}="B2";
	(8.6,5)*{\bullet}="A";
	(10.6,3.84)*{}="A1";
	(8.6,7)*{}="A2";
	(-8.6,5)*{\bullet}="C";
	(-10.6,3.84)*{}="C1";
	(-8.6,7)*{}="C2";
	(8.6,-5)*{\bullet}="D";
	(10.6,-3.84)*{}="D1";
	(8.6,-7)*{}="D2";
	(-8.6,-5)*{\bullet}="F";
	(-10.6,-3.84)*{}="F1";
	(-8.6,-7)*{}="F2";
	(0,-10)*{\bullet}="E";
	(2,-11.16)*{}="E1";
	(-2,-11.16)*{}="E2";
	"B1";"C1"**\dir{-};
	"B2";"A1"**\dir{-};
	"F1";"E1"**\dir{-};
	"D1";"E2"**\dir{-};
	"A2";"D2"**\dir{-};
	"C2";"F2"**\dir{-};
	"A"+(2,2)*{c};
	"B"+(0,3)*{b};
	"C"+(-2,2)*{a};
	"D"+(2,-2)*{d};
	"E"+(0,-3)*{e};
	"F"+(-2,-2)*{f};
\endxy
\]
There is no line containing a pair of points here except when a
line is actually shown, and more generally there are no `shortcuts'
beyond what is shown.  For example, we go from $a$ to $b$ by following
just one line, but it takes two to get from $a$ to $c$, and three to get
from $a$ to $d$.  

For a nice introduction to these ideas, see the paper by Betty
Salzberg \cite{Salzberg}.  Among other things, she notes that the
points and lines in the incidence geometry of the split real form of
$\G_2$ correspond to 1- and 2-dimensional null subalgebras of the
imaginary split octonions.  This was shown by Tits in 1955 \cite{TitsG2}.

In 1993, Robert Bryant and Lucas Hsu \cite{BryantHsu} gave a detailed
treatment of curves in manifolds equipped with 2-dimensional
distributions, greatly extending the work of Cartan.  They showed how
the space of configurations of one surface rolling on another fits into
this framework.  However, Igor Zelenko may have been the first to
explicitly mention a ball rolling on another ball in this context, and
to note that something special happens when their ratio of radii is 3
or $1/3$.  In a 2005 paper \cite{Zelenko}, he considered an invariant
of $(2,3,5)$ distributions.  He calculated it for the distribution
arising from a ball rolling on a larger ball and showed it equals zero 
in these cases.

In 2006, Bor and Montgomery's paper ``$\G_2$ and the `rolling
distribution'" put many of the pieces together \cite{BorMontgomery}.
They studied the $(2,3,5)$ distribution on $S^2 \times \SO(3)$ coming
from a ball of radius 1 rolling on a ball of radius $R$, and proved a
theorem which they credit to Robert Bryant.  First, passing to the
double cover, they showed the corresponding distribution on $S^2
\times \SU(2)$ has a symmetry group whose identity component contains
the split real form of $\G_2$ when $R = 3$ or $1/3$.  Second, they
showed this action does not descend to original rolling ball
configuration space $S^2 \times \SO(3)$.  Third, they showed that for
any other value of $R$ except $R = 1$, the symmetry group is
isomorphic to $\SU(2) \times \SU(2)/\pm(1,1)$.  They also wrote:

\begin{quote} Despite all our efforts, the `3' of the ratio 1:3 
remains mysterious. In this article it simply arises out of the
structure constants for $G_2$ and appears in the construction of the
embedding of $\so(3) \times \so(3)$ into $\g_2$.  Algebraically
speaking, this `3' traces back to the 3 edges in $\g_2$'s Dynkin
diagram and the consequent relative positions of the long and short
roots in the root diagram for $\g_2$ which the Dynkin diagram is
encoding.

\textbf{Open problem.} Find a geometric or dynamical interpretation
for the `3' of the 3:1 ratio.
\end{quote}

While Bor and Montgomery's paper goes into considerable detail about
the connection with split octonions, most of their work uses the now
standard technology of semisimple Lie algebras: roots, weights and the
like.  In 2006 Sagerschnig \cite{Sagerschnig} described the incidence
geometry of $\G_2$ using the split octonions, and in 2008, Agrachev
wrote a paper entitled ``Rolling balls and octonions''.  He emphasizes
that the double cover $S^2 \times \SU(2)$ can be identified with the
double cover of what we are calling $\PC$, the projectivization of
the space $C$ of null vectors in the imaginary split octonions.  He
then shows that given a point $\langle x \rangle \in \PC$, the set of
points $\langle y \rangle$ connected to $\langle x \rangle$ by a
single roll is the annihilator
\[           \{ x \in \I : y x = 0 \} \]
where $\I$ is the space of imaginary split octonions.

This sketch of the history is incomplete in many ways.  For more
details, try Agricola's essay \cite{Agricola} on the history of $\G_2$
and Robert Bryant's lecture about Cartan's work on simple Lie groups
of rank two \cite{Bryant}.  Aroldo Kaplan's review article
``Quaternions and octonions in mechanics'' is also very helpful
\cite{Kaplan}: it emphasizes the role that quaternions play in
describing rotations, and the way an imaginary split octonion is built
from an imaginary quaternion and a quaternion.  We take advantage of
this---and indeed most the previous work we have mentioned!---in 
what follows.  

\section{The rolling ball}

Our goal is to understand $\G_2$ in terms of a rolling ball.  It is
\emph{almost} true that split real form of $\G_2$ is the symmetry
group of a ball of radius 1 rolling on a fixed ball 3 times as large
without slipping or twisting.  In fact we must pass to the double
cover of the rolling ball system, but this is almost as nice: it is a
kind of `rolling spinor'.

Before we talk about the rolling spinor, let us introduce the
incidence geometry of the ordinary rolling ball.  This differs from
the usual approach to thinking of the rolling ball as a physical
system with a constraint, but it is equivalent.  There is an incidence
geometry where:
\begin{itemize}
	\item Points are configurations of a ball of radius 1 touching a 
              fixed ball of radius $R$.
	\item Lines are trajectories of the ball of radius 1 rolling without
		slipping or twisting along great circles on the fixed ball of
		radius $R$.
\end{itemize}
We call the ball of radius 1 the \define{rolling} ball, and the ball of
radius $R$ the \define{fixed} ball.  

To specify a point in this incidence geometry, we can give a point $x
\in S^2$ on the unit sphere, together with a rotation $g \in \SO(3)$.
Physically, $Rx \in \R^3$ is the point of contact where the rolling
ball touches the fixed ball, while $g$ tells us the orientation of the
rolling ball, or more precisely how to obtain its orientation from
some fixed, standard orientation.  Thus we define the space of points
in this incidence geometry to be $S^2 \times \SO(3)$.  This space is
independent of the radius $R$, but the lines in this space depend on
$R$.  To see how we should define them, it helps to reason physically.

We begin with the assumption that, since the rolling ball is not
allowed to slip or twist as it rolls, the point of contact traces
paths of equal arclength on the fixed and rolling balls.  In a picture:

\begin{center}
	\includegraphics[scale=0.5]{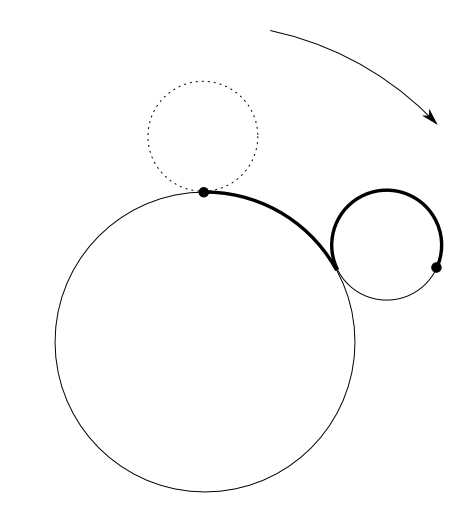}
\end{center}

Now let us quantify this.  Begin with a configuration in which the
rolling ball sits at the \define{north pole}, $(0,0,R) \in \R^3$, of
the fixed ball, and let it roll to a new configuration on a great
circle passing through the north pole, sweeping out a central angle
$\Phi$ in the process.  The point of contact thus traces out a path of
arclength $R \Phi$. As the rolling ball turns, its initial point of
contact sweeps out an angle of $\phi$ relative to the line segment
connecting the centers of both balls.  In a picture:

\begin{center}
\begin{overpic}[scale=0.5]{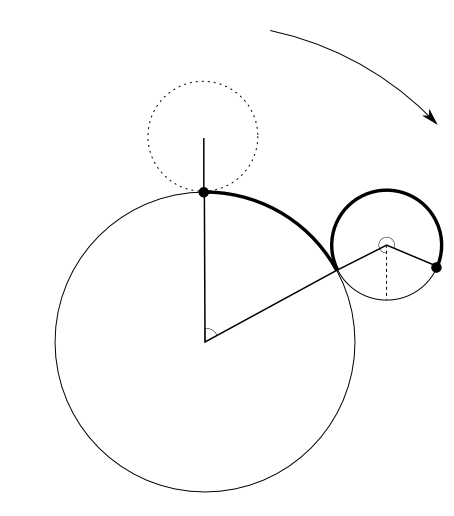}
       \put(42,38){$\Phi$}
       \put(71,47){$\Phi$}
       \put(74,56){$\phi$}
\end{overpic}
\end{center}

By assumption, the distances traced out by the point of contact on the
fixed and rolling balls are equal, and these are:
\[ R \Phi = \phi , \]
since the rolling ball has unit radius.  But because the frame of the
rolling ball has itself rotated by angle $\Phi$ in the frame of the 
fixed ball, the rolling ball has turned by an angle:
\[ \phi + \Phi = (R+1)\Phi . \]
So: in each revolution around the fixed ball, the rolling ball turns
$R+1$ times!  We urge the reader to check this directly for the case
$R = 1$ using two coins of the same sort.  As one coin rolls around
the other without slipping or twisting, it turns around twice.

This reasoning makes it natural to define the rolling
trajectories using a parameterization.  Let $u$ and $v$ be 
orthogonal unit vectors in $\R^3$.  They both lie on $S^2$, and on the
great circle parameterized by
\[ \cos(\Phi) u + \sin(\Phi) v \]
where $\Phi \in \R$.  If the rolling ball starts at $u$ in the standard
configuration, then when it rolls to $\cos(\Phi) u + \sin(\Phi) v$, it rotates about
the axis $u \times v$ by the angle $(R+1)\Phi$.  Writing $\Ro(w,\alpha)$ for the
right-handed rotation by an angle $\alpha$ about the unit vector $w$, the rolling
trajectory is
\[
\{ (\cos(\Phi) u + \sin(\Phi) v, \; \Ro(u \times v, (1+R) \Phi)) 
: \Phi \in \R \} \; \subset S^2 \; \times \SO(3) .
\]

More generally, the rolling ball may be rotated by some arbitary
element $g \in \SO(3)$ when it starts its trajectory.  Then the
rolling trajectory will be
\begin{equation}
\label{trajectory}
L = \{ (\cos(\Phi) u + \sin(\Phi) v, \; \Ro(u \times v, (1+R) \Phi)g ) 
: \; \Phi \in \R \} \; \subset \; S^2 \times \SO(3)
\end{equation}
We define a \define{line} in $S^2 \times \SO(3)$ to be any subset of
this form.  Of course this notion of line depends on $R$.  Note that
different choices of $u, v$ and $g$ may give different
parametrizations of the same line, since a rolling motion may start at
any point along a given line. In fact the space of lines is
5-dimensional: two dimensions for the choice of our starting point $u
\in S^2$, one dimension for the choice of $v \in S^2$ orthogonal to
$u$, determining the direction in which to roll, and three dimensions
for the choice of starting orientation $g \in \SO(3)$, minus one
dimension of redundancy since our starting point on the line was arbitrary.

\section{The rolling spinor}
\label{rolling_spinor}

We now consider a situation where the rolling ball behaves like
a spinor, in that it must make two whole turns instead of one
to return to its original orientation.  Technically this means
replacing the rotation group $\SO(3)$ by its double cover, the group
$\SU(2)$.  Since $\SU(2)$ can be seen as the group of unit 
quaternions, this brings quaternions into the game---and the split
octonions follow soon after!

We begin with a lightning review of quaternions.  Recall that the
\define{quaternions}
\[   \H = \{ a + bi + cj + dk : \; a,b,c,d \in \R \} \]
form a real associative algebra with product specified
by Hamilton's formula:
\[    i^2 = j^2 = k^2 = ijk = -1 . \]
The \define{conjugate} of a quaternion $x = a + bi + cj + dk$ is
defined to be $\overline{x} = a - bi - cj - dk$, and its \define{norm}
$|x|$ is defined by
\[  |x|^2 = x \overline{x} = \overline{x} x  = a^2 + b^2 + c^2 + d^2. \]
The quaternions are a \define{normed division algebra}, meaning that
they obey
\[    |xy| = |x| |y|  \]
for all $x,y \in \H$.  This implies that the quaternions of norm 1
form a group under multiplication.  This group is isomorphic to $\SU(2)$,
so indulging in a slight abuse of notation we simply write
\[    \SU(2) = \{  q \in \H : \; |q| = 1 \} .\]
Similarly, we can identify the \define{imaginary quaternions}
\[ 
      \Im(\H) = \{ x \in \H : \; \overline{x} = -x \} 
\]
with $\R^3$.  The group $\SU(2)$ acts on $\Im(\H)$ via conjugation:
given $q \in \SU(2)$ and $x \in \Im(\H)$, $qxq^{-1}$ is again in $\Im
(\H)$.  This gives an action of $\SU(2)$ as rotations of $\R^3$, which
exhibits $\SU(2)$ as a double cover of $\SO(3)$.

We can now define a spinorial version of the rolling ball incidence
geometry discussed in the last section.  We define the space of points
in the spinorial incidence geometry to be $S^2 \times \SU(2)$.
This is a double cover, and indeed the universal cover, of the space
$S^2 \times \SO(3)$ considered in the previous section.  So, we define
a \define{line} in $S^2 \times \SU(2)$ to be the inverse image under the
covering map
\[   p \maps S^2 \times \SU(2) \to S^2 \times \SO(3)  \]
of a line in $S^2 \times \SO(3)$.  

We can describe these lines more explicitly using quaternions:

\begin{prop} 
\label{trajectory_2}
Any line in $S^2 \times \SU(2)$ is of the form
\[
\tilde{L} =
\{ ( e^{2 \theta w} u , \; e^{(R+1) \theta w} q) : \; \theta \in \R \}.
\]
for some orthogonal unit vectors $u, w \in \Im(\H)$ and some
$q \in \SU(2)$.  
\end{prop}

\begin{proof} First, remember Equation \ref{trajectory}, which describes 
any line $L \subset S^2 \times \SO(3)$:
\[
L = \{ (\cos(\Phi) u + \sin(\Phi) v, \; \Ro(u \times v, (1+R) \Phi)g) 
: \; \Phi \in \R \} \; \subset \; S^2 \times \SO(3)
\]
in terms of orthogonal unit vectors in $u, v \in \R^3$ and a
rotation $g \in \SO(3)$.  To lift this line to $S^2 \times \SU(2)$, we
must replace the rotation $g$ by a unit quaternion $q$ that maps down
to that rotation (there are two choices).  Similarly, we must replace
$\Ro(u \times v, (1+R) \Phi)$ by a unit quaternion that maps down to
this rotation.  The double cover $\SU(2) \to \SO(3)$ acts as follows:
\[       e^{\theta w/2} \mapsto \Ro(w, \theta) \]
for any unit vector $w \in \Im(\H) \cong \R^3$ and any angle $\theta \in \R$.  
Thus, the inverse image of the line $L$ under the map $p$ is
\[
\tilde{L} = 
\{ (\cos(\Phi) u + \sin(\Phi) v, 
\; e^{\frac{R+1}{2} \Phi (u \times v)} q) 
: \; \Phi \in \R \} \; \subset \; S^2 \times \SU(2)
\]
We can simplify this expression a bit by writing $u \times v$ as
$w$, so that $u,v,w$ is a right-handed orthonormal triple in
$\Im(\H)$.  Then
\[   \cos(\Phi) u + \sin(\Phi) v = e^{\frac{1}{2} \Phi w} u 
e^{-\frac{1}{2} \Phi w} \]
since this vector is obtained by rotating $u$ by an angle $\Phi$
around the axis $w$.  However, since $u$ and $w$ are orthogonal
imaginary quaternions, they anticommute, so we obtain
\[   \cos(\Phi) u + \sin(\Phi) v = e^{\Phi w} u .\]
Thus any line in $S^2 \times \SU(2)$ is of the form
\[
\tilde{L} =
\{ ( e^{\Phi w} u ,
 \; e^{\frac{R+1}{2} \Phi w} \, q) 
: \; \Phi \in \R \} 
\]
Even better, set $\theta = \Phi/2$.  Then we have
\[  \tilde{L} =
\{ ( e^{2\theta w} u , 
 \; e^{(R+1) \theta w} q) 
: \; \theta \in \R \} .
\qedhere \]
\end{proof}

\section{The rolling spinor on a projective plane}

We now consider a spinor rolling on a projective plane.  In other
words, we switch from studying lines on $S^2 \times \SU(2)$ to
studying lines on $\RP^2 \times \SU(2)$.  As before, these lines
depend on the radius $R$ of the rolling ball.

There is a double cover 
\[   q \maps S^2 \times \SU(2) \to \RP^2 \times \SU(2) \]
Since $S^2 \times \SU(2)$ was introduced as a double cover of the $S^2
\times \SO(3)$ in the first place, it may seem perverse to introduce
another space having $S^2 \times \SU(2)$ as a double cover:
\[\xymatrix{
& S^2 \times \SU(2)   \ar[dl]_p \ar[dr]^q 
\\ S^2 \times \SO(3) & & \RP^2 \times \SU(2) 
} \] 
However, $\RP^2 \times \SU(2)$ is not diffeomorphic to the original
rolling ball configuration space $S^2 \times \SO(3)$.  More
importantly, it \emph{is} diffeomorphic to the space of null lines
through the origin in $\Im(\H) \oplus \H$, a 7-dimensional vector
space equipped with a quadratic fom of signature $(3,4)$.

To see this, first recall from Section \ref{rolling_spinor} that a
point in $S^2 \times \SU(2)$ is a pair $(v,q)$ where $v$ is a unit
imaginary quaternion and $q$ is a unit quaternion.  So, a point in
$\RP^2 \times \SU(2)$ is an equivalence class consisting of two points
in $S^2 \times \SU(2)$, namely $(v,q)$ and $(-v,q)$.  We write this
equivalence class as $(\pm v, q)$.

We can describe a null line through the origin in $\Im(\H) \oplus \H$
in a very similar way.  First, note that $\Im(\H) \oplus \H$ has
a quadratic form $Q$ given by
\[ Q(a,b) = |a|^2 - |b|^2. \]
A \define{null vector} in this space is one with $Q(x) = 0$.
Let $C$ be the set of null vectors:
\[     C = \{ x \in \Im(\H) \oplus \H : \; Q(x) = 0 \} .\]
This is what physicists might call a \define{lightcone}.  However,
the signature of $Q$ is $(3,4)$, so this lightcone lives in an 
exotic spacetime with 3 time dimensions and 4 space dimensions.

Let $\PC$ be the corresponding \define{projective lightcone}:
\[    \PC = \{x \in C : x \ne 0\}/\R^*  \]
where $\R^*$, the group of nonzero real numbers, acts by rescaling the
cone $C$.  A point in $\PC$ can be identified with a 1-dimensional
\textbf{null subspace} of $\Im(\H) \oplus \H$, by which we mean a
subspace consisting entirely of null vectors.  We can write any
1-dimensional null subspace as $\langle x \rangle$, the span of any
nonzero null vector $x$ lying in that subspace.  We can always
normalize $x = (v,q)$ so that
\[   |v|^2 = |q|^2 = 1 .\]
The space of vectors $x$ of this type is $S^2 \times \SU(2)$, and two
such vectors $x$ and $x'$ span the same subspace if and only if $x' = \pm
x$.  So, we shall think of a point in $\PC$ as an equivalence class
of points in $S^2 \times \SU(2)$ consisting of the points $(v,q)$ and
$(-v,-q)$.  We write this equivalence class as $\pm (v,q)$.

\begin{prop}
There is a diffeomorphism 
\[     \tau \maps \RP^2 \times \SU(2) \to \PC \]
sending $(\pm v, q)$ to $\pm (v, v q)$.  
\end{prop}

\begin{proof} 
First note that $\tau$ is well-defined: reversing the sign of 
$v$ reverses the sign of $(v,v q)$.   Next note that $\tau$ has
a well-defined inverse, sending $\pm (v, q)$ to $(\pm v , v^{-1}q)$.
It is easy to check that both $\tau$ and its inverse are smooth.
\end{proof}

There is a double cover
\[   q \maps S^2 \times \SU(2) \to \RP^2 \times \SU(2) \]
sending $(v,q)$ to the equivalence class $(\pm v, q)$.  We define a
\define{line} in $\RP^2 \times \SU(2)$ to be the image of a line in
$S^2 \times \SU(2)$ under this map $q$.  We then define a
\define{line} in $\PC$ to be the image of a line in $\RP^2 \times
\SU(2)$ under the diffeomorphism $\tau$.

In short, we can think of configurations and trajectories of a rolling
spinorial ball on a projective plane as points and `lines' in $\PC$.
But this concept of `line' depends on the radius $R$ of the ball.
When $R = 3$, these lines have a wonderful property: they come from
projectivizing planes inside the lightcone $C$.  To see this, we need
an explicit desciption of these lines:

\begin{prop}
\label{trajectory_3}
Fixing the radius $R$, every line in $\PC$ is of the form
\[ L=
\{ \pm(e^{2 \theta w} u, 
 e^{-(R-1) \theta w} u q) : \; \theta \in \R \} 
\; \subset \; \PC .  
\]
for some orthogonal unit vectors $u, w \in \Im(\H)$ and some
$q \in \SU(2)$.  
\end{prop}

\begin{proof} 
Recall from Proposition \ref{trajectory_2} that any line in $S^2
\times \SU(2)$ is of the form
\[
\{ ( e^{2 \theta w} u , \; e^{(R+1) \theta w} q) : \; \theta \in \R \} 
\]
where $u, w$ are orthogonal unit vectors in $\Im(\H)$ and $q \in \SU(2)$.
Thus, any line in $\RP^2 \times \SU(2)$ is of the form
\[
\{ (\pm e^{2 \theta w} u , \; e^{(R+1) \theta w} q) : \; \theta \in \R \} 
\]
and applying the map $\tau$, any line in $\PC$ is of the form
\[
\{ \pm(e^{2 \theta w} u, 
e^{2 \theta w} u \; e^{(R+1) \theta w} q) : \; \theta \in \R \} .
\]
Since $u$ and $w$ are orthogonal imaginary quaternions,
they anticommute, so we may rewrite this as
\[
\{ \pm(e^{2 \theta w} u, 
 e^{-(R-1) \theta w} u q) : \; \theta \in \R \} .
 \qedhere
\]
\end{proof}

Suppose $X \subset \Im(\H) \oplus \H$ is a 2-dimensional null
subspace.  Then we can projectivize it and get a curve in $\PC$:
\[    \Proj X = \{x \in X : x \ne 0\}/\R^* . \]
When $R = 3$, \emph{and only then}, every line in $\PC$ is a
curve of this kind:

\begin{thm} 
\label{thm:nullsubspace}
If and only if $R = 3$, every line in $\PC$ is the
projectivization of a 2-dimensional null subspace of $\Im(\H) \oplus \H$.
\end{thm}

\begin{proof}
To prove this, it helps to polarize $Q$ and introduce a \define{dot
product} on $\Im(\H) \oplus \H$, namely the unique symmetric bilinear
form such that
\[ x \cdot x = Q(x) . \]
A subspace $X \subset \Im(\H) \oplus \H$ is null precisely when
this bilinear form vanishes on $X$.  We will need an explicit
formula for this bilinear form:
\[  (a,b) \cdot (c,d) = a \cdot c - b \cdot d  \]
where at right $\cdot$ is the usual dot product on $\H$:
\[ a \cdot b = \Re(\overline{a} b) . \]
We will also need to recall that the dot product of imaginary
quaternions is the same as the usual dot product on $\R^3$. 

Now consider an arbitrary line $L \subset \PC$.  By Proposition
\ref{trajectory_3} this is of the form
\[
L = \{ \pm(e^{2 \theta w} u,  e^{-(R-1) \theta w} u q) 
: \; \theta \in \R \} 
\]
for some orthogonal unit vectors $u, w \in \Im(\H)$ and $q \in
\SU(2)$.  Assume that $L$ is the projectivization of some null
subspace $X \subset \Im(\H) \oplus \H$.  Then every pair of vectors 
$x, y \in X$ must have $x \cdot y = 0$.  We now show that this
constrains $R$ to equal 3.

Indeed, letting $\theta = 0$, one such vector is 
\[   x = (u, u q) , \]
while letting $\theta$ be arbitrary, another is 
\[   y = (e^{2\theta w} u, e^{(1-R)\theta w} u q) .\]
We have
\[  \begin{array}{ccl}
 x \cdot y &=& 
u \cdot e^{2\theta w} u - u q \cdot e^{(1-R)\theta w} u q \\
&=& u \cdot e^{2\theta w} u - u \cdot e^{(1-R)\theta w} u 
\end{array}
\]
where in the second step we note that right multiplication by a
unit quaternion preserves the dot product.  Since
$e^{2\theta w} u$ is $u$ rotated by an angle $2\theta$ about the 
$w$ axis, which is orthogonal to $u$, we have
\[    u \cdot e^{2\theta w} u  = \cos (2 \theta) .\]
Similarly 
\[   u \cdot e^{(1-R)\theta w} u = \cos (2(1-R)\theta) .\]
To ensure $x \cdot y = 0$, we thus need
\[ \cos(2\theta) = \cos((1-R) \theta) . \]
This must hold for all $\theta$, so we need $1 - R = \pm 2$. 
Since we are assuming the rolling ball has positive radius, we
conclude $R = 3$. 

On the other hand, suppose that $R = 3$. Then any line in $\PC$ 
has the form:
\[
L = \{ \pm(e^{2 \theta w} u,  e^{-2 \theta w} u q) 
: \; \theta \in \R \} .
\]
Expanding the exponentials:
\[ 
\begin{array}{ccl}
 (e^{2 \theta w} u,  e^{-2 \theta w} u q) &=& 
(\cos(2 \theta) u + \sin(2\theta) w u, \; \cos(2\theta) u q - 
\sin(2 \theta) w u q) \\
 & = & \cos(2 \theta) (u,u q) + \sin(2\theta) ( w u, -w u q) .
\end{array}
\]
we see that this vector lies in the 2-dimensional null
subspace spanned by the orthogonal null vectors $(u, u q)$
and $(w u, -w u q)$.  Thus, $L$ is the projectivization
of a 2-dimensional null subspace.
\end{proof}

Assume $R = 3$.  Then every line in $\PC$ is the projectivization of
a 2-dimensional null subspace.  But the converse is false: not every
2-dimensional null subspace gives a line in $\PC$ when we
projectivize it.  Which ones do?  The answer requires us to introduce
the split octonions!  As we shall see in the next section, it is
precisely the 2-dimensional `null subalgebras' of the split octonions
that give lines in $\PC$.

\section{Split octonions and the rolling ball}

We have seen that the configuration space for a rolling spinorial ball
on a projective plane is the projective lightcone $\PC$.
We have also seen that the lines in this space are especially
nice when the ratio of radii is 1:3.  To go further, we
now identify $\Im(\H) \oplus \H$ with the imaginary split octonions.
This lets us prove that when the ratio of radii is 1:3, lines in
$\PC$ can be defined using the algebra structure of the split
octonions.  Thus, automorphisms of the split octonions act to give
symmetries of the configuration space that map lines to lines.  This
symmetry group is $\G_2'$, the split real form of $\G_2$.

Every simple Lie group comes in a number of forms: up to covers, there
is a unique complex form, as well as a compact real form and a split
real fom.  Some groups have additional real forms: any real Lie group
whose complexification is the complex form will do. For $\G_2$,
however, there are only the three forms.  Each is the automorphism
group of some 8-dimensional composition algebra---in other words, some
form of the octonions.

A \define{composition algebra} $A$ is a possibly nonassociative algebra
with a multiplicative unit 1 and a nondegenerate quadratic form $Q$
satisfying 
\[ Q(xy) = Q(x) Q(y)  \]
for all $x,y \in A$.  This concept makes sense over any field.  Right
now we only need real composition algebras, but in the next section we
will need a complex one.  

Up to isomorphism, there are just two 8-dimensional real composition
algebras, and their automorphism groups give the two real forms of
$\G_2$:
\begin{itemize}
\item 
The \define{octonions}, $\O$, is the vector space $\H \oplus \H$
with the product
\[ (a,b)(c,d) = (ac - d\overline{b}, \overline{a}d + cb) . \]
This becomes a composition algebra with the positive
definite quadratic form given by
\[ Q(a,b) = |a|^2 + |b|^2. \]
The automorphism group of $\O$ is the compact 
real form of $\G_2$, which we denote simply as $\G_2$.  This group
is simply-connected and has trivial center.
\item 
The \define{split octonions}, $\O'$, is the vector space $\H \oplus \H$
with the product
\[ (a,b)(c,d) = (ac + d\overline{b}, \overline{a}d + cb) . \]
This becomes a composition algebra with the nondegenerate
quadratic form of signature $(4,4)$ given by
\[ Q(a,b) = |a|^2 - |b|^2. \]
The automorphism group of $\O'$ is the split
real form of $\G_2$, which we denote as $\G_2'$.  More precisely,
this is the adjoint split real form, which has fundamental group $\Z_2$
and trivial center.  There is also a simply-connected split real form
with center $\Z_2$. 
\end{itemize}

It is the split octonions, $\O'$, that are the most closely connected
to the rolling ball.  As with $\H$, the quadratic form on $\O'$ can 
also be defined using conjugation. If we take
\[ \overline{(a,b)} = (\overline{a}, -b) . \]
then we can check that
\[ Q(x) = x\overline{x} = \overline{x} x. \]
This conjugation satisfies some the same nice properties as quaternionic
conjugation:
\[ \overline{\overline{x}} = x, \quad \overline{xy} = \overline{y} 
\, \overline{x} . \]
We define the \define{imaginary split octonions} by
\[ \I = \{ x \in \O' : \; \overline{x} = -x \} = \Im(\H) \oplus \H . \]
Since conjugation in $\O'$ is invariant under all the automorphisms of
$\O'$, the same is true of the subspace $\I$, so we obtain a
7-dimensional representation of $\G'_2$.  This is well-known to be an
irreducible representation.  The quadratic form $Q$ has signature
$(3,4)$ when restricted to $\I$.

As promised at the start of this section, the lightcone in $\Im(\H) \oplus
\H$ now lives in $\I$, the imaginary split octonions: 
\[ C \subset \I . \]
Moreover, because $\G'_2$ preserves the quadratic form on $\I$, it
acts on $C$, as well as its projectivization:
\[    \PC = \{x \in C : x \ne 0\}/\R^*  .\]
We have already seen how to view this space as the configuration space
of a spinor rolling on a projective plane, and how to describe the
rolling trajectories in that configuration space for any ratio of
radii.  We now show, when that ratio is 1:3, the action of $\G'_2$
preserves these rolling trajectories.

We define a \define{null subalgebra} of $\O'$ to be a vector subspace
$V \subset \O'$ on which the product vanishes.  In other words, $V$
is closed under addition, scalar multiplication by real numbers, and
$xy = 0$ whenever $x,y \in V$.  Such a subalgebra clearly does
\emph{not} contain the unit $1 \in \O'$.  In fact, because the square
of an element with nonzero real part cannot vanish, any null
subalgebra must be purely imaginary.  It must also be a null subspace
of the imaginary split octonions, since $Q(x) = x\overline{x} = -x^2 = 0$
for an imaginary split octonion in a null subalgebra.  Thus, the
projectivization of a null subalgebra gives a subset of the projective
lightcone, $\PC$.

\begin{thm} 
\label{thm:lines_and_2d_null_subalgebras}
Suppose $R = 3$.  Then any line in $\PC$ is the projectivization of
some 2d null subalgebra of $\O'$, and conversely, the projectivization
of any 2d null subalgebra gives a line in $\PC$.
\end{thm}

\begin{proof}
Let $L$ be a line in $\PC$.  By Proposition \ref{trajectory_3} this is
of the form
\[
L = \{ \pm(e^{2 \theta w} u,  e^{-2\theta w} u q) 
: \; \theta \in \R \} 
\; \subset \; \PC   
\]
when $R = 3$.  By Theorem \ref{thm:nullsubspace}, $L$ is a
projectivization of a 2-dimensional null subspace $X \subset \Im(\H)
\oplus \H$.  This subspace is spanned by any two linearly independent
vectors in $X$, so putting $\theta = 0$ and $\theta = \frac{\pi}{4}$
in our formula for $L$, we have:
\[ X = \langle (u, u q), (w u, -w u q) \rangle . \]
We claim that $X$ is a null subalgebra.  To prove this, it suffices to
check that the product of any two vectors in this basis vanishes.
Because both vectors are null and imaginary, their squares
automatically vanish:
\[ Q(u, u q) = (u, u q) \overline{(u, u q)} = -(u, u q)^2 = 0, \]
and similarly for $(w u, -w u q)$. It thus remains to show that their
product vanishes:
\[
\begin{array}{ccl}
(u, u q) (w u, -w u q) &=& 
(u w u + (-w u q)\overline{u q}, \overline{u} (-w u q) + w u u q) \\
 &=& (u w u - w , u w u q - w q) \\
 &=& 0 , \\
\end{array}
\]
where we used the fact that the unit imaginary quaternion $u$ anticommutes
with $w$.  Thus $X$ is a 2-dimensional null subalgebra.

On the other hand, given a 2-dimensional null subalgebra $X$, we wish
to show that its projectivization gives a line in $\PC$.  To prove this
it suffices to show that $X$ has the form
\[ X = \langle (u, u q), (w u, -w u q) \rangle . \]
for some orthogonal unit imaginary quaternions $u$ and $w$ and unit
quaternion $q$, since then reversing the calculation above shows that the
projectivization of $X$ is a curve in $\PC$ of this form:
\[
L = \{ \pm(e^{2 \theta w} u,  e^{-2\theta w} u q) 
: \; \theta \in \R \} .
\]

So, fix any nonzero vector $x \in X$. It is easy to check that $x =
(u, u q)$ for some imaginary quaternion $u$ and quaternion $q$. By
rescaling, we can assume $u$ has unit length, forcing $q$ to also
have unit length, since $x$ is null.

Next choose any linearly independent vector $y = (v,v') \in X$. By
subtracting a multiple of $x$ from $y$, we can ensure the first
component of $y$ is orthogonal to the first component of $x$.  By
rescaling the result, we can also assume that $v$ and $v'$ both have
unit length.  We can thus obtain $v$ from $u$ by multiplication by a
unit quaternion orthogonal to them both, say $w$:
\[ v = w u . \]
In summary, we have:
\[ X = \langle (u, u q), (w u, v') \rangle . \]
Finally, because $X$ is a null subalgebra, we must have $xy = 0$, and
this forces $v' = -w u q$.  Indeed:
\[ xy = (u, u q) (w u, v') = 
(u w u + v' \overline{u q}, \, \overline{u} v' + w u u q) , \]
and a quick calculation shows this vanishes if and only if $v' = - w u q$, 
as
desired.
\end{proof}

\begin{cor}
When $R = 3$, the group $\G'_2$ acts on $\PC$ in a way that maps
lines to lines.
\end{cor}

Knowing that the lines in $\PC$ correspond to 2-dimensional null
subalgebras of the space $\I$ of imaginary split octonions, we can use
operations on $\I$ to study the incidence geometry of $\PC$.  The
concepts here we also apply to the complexification $\PC^\C$, which
we study in Section \ref{cross product}.  Thus, we state them in a
way that applies to both cases.

First, because the lines in $\PC$ are projectivizations of 2d null
subalgebras, it will be very helpful for us to understand the
\define{annihilator} of a null imaginary split octonion, $x$:
\[ \Ann_x = \{ y \in \I : \; yx = 0 \} . \]
This subspace of $\I$ is intimately related to set of lines through $\langle x
\rangle \in \PC$: any point $y \in \Ann_x$ linearly independent of $x$ will
span a 2d null subalgebra with $x$, which in turn projectivizes to give a line
through $\langle x \rangle$. So: understanding the annihilator is crucial for
understanding how other points are connected to $\langle x \rangle$ via lines,
and this will move to the center of our focus in Section \ref{cross product}.

\begin{prop}
	\label{lem:annihilator}
	Let $x \in C$ be a nonzero null vector. Then we have:
	\begin{enumerate}
		\item $\Ann_x$ is a null subspace.
		\item Any two elements of $\Ann_x$ anticommute.
		\item $\Ann_x$ is three-dimensional.
	\end{enumerate}
\end{prop}

\begin{proof}

First we show that $\Ann_x$ is a null subspace.  Consider two elements
$y, y' \in \Ann_x$. In fact, because the dot product of two imaginary
split octonions is proportional to their anticommutator:
\[ y \cdot y' = -\frac{1}{2}(y y' + y' y) , \]
we can show that $y$ and $y'$ are orthogonal and anticommute in one blow,
proving parts 1 and 2.

Indeed, the real number $-2(y \cdot y')$ vanishes if and only if its
product with a nonzero vector vanishes. We consider its product with
$x$, since $y$ and $y'$ annihilate $x$ by definition:
\[ -2(y \cdot y')x 
= (yy')x + (y'y)x = y(y'x) + y'(yx) + [y,y',x] + [y',y,x] = 0 . \]
where $[x,y,z] = (xy)z - x(zy)$ is the \define{associator}.  The first
two terms are zero because $y$ and $y'$ annihilate $x$.  The last two
terms cancel since the associator is antisymmetric in its three
arguments, thanks to the fact that the split octonions are alternative
\cite{Schafer}.

To prove part 3 and show that $\Ann_x$ is 3-dimensional, write the
imaginary split octonion $x$ as a pair $(u,q) \in \Im(\H) \oplus \H$.
Since rescaling the null vector $x$ does not change $\Ann_x$, we may
assume without loss of generality that it is normalized so that $u \,
\overline{u} = q \, \overline{q} = 1$.  We shall show that $\Ann_x$ is
isomorphic to the vector space of imaginary quaternions, $\Im(\H)$.
To do this, let $y$ be any element of $\Ann_x$, and write it as a pair
$(c,d)$.  Then
\[   xy = (uc + d\overline{q}, \overline{u}d + cq) . \]
This expression vanishes if and only if $d = -ucq$.  Thus $y =
(c,-ucq)$, and the map
\[ 
\begin{array}{cccl}
f \maps & \Im(\H) & \to     & \Ann_x \\
& c & \mapsto & (c,-ucq) 
\end{array}
\]
is an isomorphism of vector spaces.  \end{proof}

For some familiar geometries, such as that of a projective space, any
two points are connected by a line. This is not true for $\PC$,
however.  We can see this using the rolling ball description: as the
ball rolls along a great circle from one point of contact to another,
it rotates in a way determined by the constraint of rolling without
slipping or twisting.  If our initial and final configurations do not
differ by this rotation, there is no way to connect them by a single
rolling motion.   In general we need multiple rolls to connect  two
configurations, so we give the following definition:

\begin{defn}
\label{defn:rolls}
We say that two points $a, b$ are \define{at most $n$ rolls away} if
there is a sequence of points $a_0 , a_1, \dots, a_n$ such that the
$a_0 = a$, $a_n = b$, and for any two consecutive points
there is a line containing those two points.  We say $a$ and $b$ are
\define{$n$ rolls away} if $n$ is the \emph{least} number for which
they are at most $n$ rolls away.
\end{defn}
\noindent

Note that because there is a line containing any point, if $a$ and $b$
are at most $n-1$ rolls away, they are also at most $n$ rolls away.
The following basic facts hold both for $\PC$ and its
complexification:

\begin{prop} 
\label{prop:rolls}
We have:
\begin{enumerate}
\setcounter{enumi}{-1}
\item Two points $a$ and $b$ are zero rolls away if and only if $a = b$. 
\item Two points $a$ and $b$ are one roll away if and only if there is a 
line containing them but $a \ne b$.
\item Two points $a$ and $c$ are two rolls away if and only if there 
exists a unique point $b$ such that:
\begin{itemize}
\item there is a line containing $a$ and $b$,
\item there is a line containing $b$ and $c$.
\end{itemize}
\end{enumerate}
\end{prop}

\begin{proof}  Part 0 is immediate from a careful reading of 
Definition \ref{defn:rolls}.  Part 1 then follows.  For part 2, first
suppose $a$ and $c$ are two rolls away.  Since they are at least two
rolls away, for some point $b$ there is a line containing $a$ and $b$
and a line containing $b$ and $c$.  We must show the point $b$ with
this property is unique.  

Suppose $b'$ were another such point. Let us write $a = \langle x
\rangle$, $b = \langle y \rangle$, $b' = \langle y' \rangle$, and $c =
\langle z \rangle$.  We know $x, z \in \Ann_y$, since $\langle x, y
\rangle$ and $\langle y, z \rangle$ are 2d null subalgebras: the 2d
null subalgebras that projectivize to give the lines joining $a$ and
$b$ and $b$ and $c$. Now, if $\langle x, y, z \rangle$ is itself
two-dimensional, then we have:
\[ \langle x, y \rangle = \langle x, y, z \rangle = \langle y,z \rangle , \]
whence $x$ and $z$ are contained in a 2d null subalgebra and $a$ and
$c$, connected by a line, are actually one roll apart. So we must have
$\langle x, y, z \rangle$ three-dimensional, and hence $x, y$ and $z$
are linearly independent. In fact, we must have:
\[ \Ann_y = \langle x, y, z \rangle , \]
since, by Proposition \ref{lem:annihilator}, $\Ann_y$ is three-dimensional.

Similarly, $\Ann_{y'} = \langle x, y', z \rangle$. In particular,
since annihilators are null subspaces by Proposition \ref{lem:annihilator},
$y'$ is orthogonal to $x$ and $z$. Moreover, since, $y$ and $y'$ both
annihilate $x$, $y$ and $y'$ are also orthogonal. Thus, $\langle x, y,
y', z \rangle$ is null, but because the maximal dimension of a null
subspace of the 7-dimensional space $\I$ is three, $y'$ must be a
linear combination of the other vectors:
\[ y' = \alpha x + \beta y + \gamma z . \]
Multiplying by $x$:
\[ x y' = \gamma xz = 0 . \]
We must have $xz \neq 0$, otherwise $a$ and $c$ are joined by the line
obtained from the 2d null subalgebra $\langle x, z \rangle$, so this
implies $\gamma = 0$. Similarly, because $y' z = 0$, we can conclude
$\alpha = 0$. Thus $y' = \beta y$. In other words, $b = \langle y
\rangle = \langle y' \rangle = b'$.

Conversely, suppose there exists a unique point $b$ such that $a$ and $b$
lie on a line and $b$ and $c$ lie on a line.  Then clearly $a$ and $c$
are at most two rolls away.  Suppose they were at most one roll away.
Then there would be a line containing $a$ and $c$.  There are
infinitely many points on this line, contradicting the uniqueness of
$b$.  Thus, $a$ and $c$ are exactly two rolls away. \end{proof}

Given nonzero $x, y \in C$, how can we tell how many rolls away
$\langle x \rangle$ is from $\langle y \rangle$?  We can use the dot
product and cross product of imaginary split octonions.  We have
already defined the \define{dot product} of split octonions by
polarizing the quadratic form $Q$:
\[       x \cdot x = Q(x) , \]
but on $\I$ it is proportional to the anticommutator:
\[        x \cdot y = - \frac{1}{2} \left( x y + y x \right), \]
as easily seen by explicit computation.  Similarly, we define the
\define{cross product} of imaginary split octonions to be half the
commutator:
\[        x \times y =  \frac{1}{2} \left( x y - y x \right). \]
For $x,y \in \I$ we have
\[         x y = x \times y - x \cdot y  \]
where $x \times y$ is an imaginary split octonion and $x \cdot y$ is 
a multiple of the identity. 

\begin{thm} 
\label{thm:incidence_real}
Suppose that $\langle x \rangle, \langle y \rangle \in \PC$.  Then:

\begin{enumerate}
\item $\langle x \rangle$ and $\langle y \rangle$ are at most one roll
away if and only if $x y = 0$, or equivalently, $x \times y = 0 $.
\item $\langle x \rangle$ and $\langle y \rangle$ are at most two rolls
away if and only if $x \cdot y$ = 0.
\item $\langle x \rangle$ and $\langle y \rangle$ are always at most
three rolls away.
\end{enumerate}

\end{thm}

\begin{proof} 
For part 1, first recall that by definition, $\langle x \rangle$ is at
most one roll away from $\langle y \rangle$ if and only if $\langle x,
y \rangle$ is a null subalgebra.  This happens if and only if $xy = 0$
and $yx = 0$.  But
\[   yx = \overline{y}\, \overline{x} = \overline{xy}  \]
since for the imaginary split octonions $x$ and $y$, we have $\overline{x} =
-x$ and $\overline{y} = -y$. Thus, it is enough to say $xy = 0$.

Next let us show that $xy = 0$ if and only if $x \times y = 0$.  If
$xy = 0$, then $yx = 0$ as well by the above calculation, so $x \times
y$, being half the commutator of $x$ and $y$, is also zero.

For the converse, suppose $x \times y = 0$.  Then $x$ and $y$ commute,
so $xy = - x \cdot y$.  Thus, it suffices to show $x \cdot y
= 0$.  Since $x \ne 0$, it is enough to show $(x \cdot y)x = 0$.  For
this we use the fact that the split octonions are
\define{alternative}: the subalgebra generated by any two elements is
associative \cite{Schafer}.  The subalgebra generated by $x$ and $y$
is thus associative and commutative, so indeed
\[   (x \cdot y) x = -\frac{1}{2}(x y + y x)x = -x^2 y = (x \cdot x)y = 0 \]
where in the last step we use the fact that $x$ is null.  

For part 2, first suppose that $\langle x\rangle$ and $\langle z
\rangle$ are at most two rolls away.  Then there is a point $\langle y
\rangle \in \PC$ that is at most one roll away from $\langle
x\rangle$ and also from $\langle z \rangle$.  Thus we know $x y = 0 =
z y$ by part 1.  We wish to conclude that $x \cdot z = 0$. But this
follows because $x, z \in \Ann_y$, and annihilators are null subspaces
by Proposition \ref{lem:annihilator}.

For the converse suppose $x \cdot z = 0$.  If $xz = 0$ we are
done, since by part 1 it follows that $x$ and $z$ are at most one roll
away.  If $xz \ne 0$ we can take $\langle xz \rangle \in \PC$, 
and we claim this point is at most one roll away from $\langle x \rangle$
and also from $\langle z \rangle$.   To check this, by part 1 it suffices
to show $x(xz) = 0$ and $(xz)z = 0$.  But since the split octonions are
alternative, we have
\[     x(xz) = x^2 z = -(x \cdot x) z = 0 \]
since $x$ is null.  Similarly $(xz) z = 0$ since $z$ is null.

For part 3, now let us show that every pair of points in $\PC$ is at most
three rolls away.  It suffices to show that given $\langle x \rangle, \langle z
\rangle \in \PC$, there exists $\langle y \rangle$ that is at most one roll
away from $\langle x \rangle$ and at most two rolls away from $\langle z
\rangle$.  Thus, by parts 1 and 2, we need to find a nonzero null imaginary
octonion $y$ with $xy = 0$ and $y \cdot z = 0$.

By Proposition \ref{lem:annihilator}, the space $\Ann_x$ of $y$ with
$xy = 0$ is 3-dimensional. Thus the linear map:
\[ \begin{array}{ccc} 
		\Ann_x & \to & \R \\
		y & \mapsto & y \cdot z 
\end{array}
\]
has at least a two-dimensional kernel, guaranteeing the existence of the
desired $y$.  \end{proof}

\section{Null triples and incidence geometry}
\label{sec:nulltriples}

Next, we shall use our octonionic description of the rolling spinor to
further investigate its incidence geometry.  To do this, we introduce
a tool we call a `null triple'.

\begin{defn} 
A \define{null triple} is an ordered triple of nonzero null imaginary
split octonions $x,y,z \in \I$, pairwise orthogonal, obeying the 
normalization condition:
\[ (x \times y) \cdot z = \frac{1}{2} . \]
\end{defn}

We shall show that any null triple generates $\I$ under the cross product, so
the action of an automorphism $g \in \G_2'$ of the split octonions is
determined by its action on a null triple.  In fact, in Theorem
\ref{thm:torsor}, we prove that the set of all null triples is a
\define{$\G_2'$-torsor}: given two null triples, there exists a unique element
of $\G_2'$ carrying the first to the second.

Null triples are well suited to the incidence geometry of the rolling
spinor because they are null, so that each member of the triple
projectivizes to give a point of $\PC$.  We shall see that these
points are all two rolls away from each other.  The relationship to
$\G_2$ runs deeper, however, as one can see by examining the cross
product multiplication table we describe below, which is conveniently
plotted as a hexagon with an extra vertex in the middle:
\[
\xy
(-21.5,12.5)*{\bullet}="X";
(0,25)*{\bullet}="XY";
(21.5,12.5)*{\bullet}="Y";
(21.5,-12.5)*{\bullet}="YZ";
(0,-25)*{\bullet}="Z";
(-21.5,-12.5)*{\bullet}="ZX";
(0,0)*{\bullet}="O";
{\ar@{->}"X";"XY"};
{\ar@{->}"XY";"Y"};
{\ar@{->}"Y";"YZ"};
{\ar@{->}"YZ";"Z"};
{\ar@{->}"Z";"ZX"};
{\ar@{->}"ZX";"X"};
{\ar@{->}"XY";"O"};
{\ar@{->}"O";"Z"};
{\ar@{->}"YZ";"O"};
{\ar@{->}"O";"X"};
{\ar@{->}"ZX";"O"};
{\ar@{->}"O";"Y"};
"X"+(-2,2)*{x};
"XY"+(0,3)*{x \times y};
"Y"+(2,2)*{y};
"YZ"+(2,-3)*{y \times z};
"Z"+(0,-3)*{z};
"ZX"+(-2,-3)*{z \times x};
"O"+(12,0)*{2 (x \times y) \times z};
\endxy
\]
The resemblance to the weight diagram of the 7-dimensional irreducible
representation $\I$ of $\G_2'$ is no accident!  Indeed, for any such
decomposition into weight spaces, three nonadjacent vertices of the outer
hexagon will be spanned by a null triple.  

We begin by showing that we can use the above hexagon to describe both
the dot and cross product in $\I$ starting with a null triple.  The
arrows on this hexagon help us keep track of the cross product.  For
convenience, we speak of `vertices' when we mean the basis vectors in
$\I$ corresponding to the seven vertices in this diagram.  For the commutative
dot product:

\begin{itemize}
\item All six outer vertices are null vectors.  
\item Opposite pairs of vertices have dot product $\frac{1}{2}$.
\item Each outer vertex is orthogonal to all the others except
its opposite. 
\item The vertex in the middle is orthogonal to all the outer
vertices, but it is not null. Instead, its dot product with itself
is $-1$.
\end{itemize}
As for the anticommutative cross product:
\begin{itemize}
\item The cross product of adjacent outer vertices is zero. 
\item For any two outer vertices that are neither adjacent nor
opposite, their cross product is given by the outer vertex between
them if they are multiplied in the order specified by the orientation of
the arrows.   For example, $(z \times x) \times (x \times y) = x$.
\item The cross product of opposite outer vertices, multiplied in the order
specified by the orientation, is half the vertex in the middle.
\item  The cross product of the vertex in the middle and an outer
vertex gives that outer vertex if they are multiplied
in the order specified by the orientation.
\end{itemize}

Now, let us prove these claims:
\begin{thm}
\label{thm:nulltriple}
Given a null triple $(x,y,z)$, the following is a basis for $\I$:
\[ x, \quad y, \quad z, \quad x \times y, \quad y \times z, \quad z \times x, \quad 2(x \times y) \times z . \]
In terms of this basis, the dot and cross product on $\I$ take the form
described above.
\end{thm}
\begin{proof}

We start by computing the dot product of outer vertices: that is, vectors
corresponding to vertices on the outside of the hexagon.  Then we compute the
cross product of outer vertices.  Next we compute the dot and cross product of
all the outer vertices with the middle vertex, and the dot product of the
middle vertex with itself.  Finally, we verify that the above vectors are
indeed a basis.

First, let us check that each outer vertex is a null vector. For $x$, $y$, and
$z$ this is true by the definition of a null triple, so we need only
check it for the other three.  It suffices to consider $x \times y$.
Since $x$ and $y$ are orthogonal we have $x \times y = xy = -yx$, so
using the alternative law we have
\[  
(x \times y)(x \times y) = -(xy)(yx) = (x(y^2))x = 0 .\]
This implies that the dot product of $x \times y$ with itself vanishes.

Next we check that both the dot and cross product of two adjacent outer vertices
vanishes.  It suffices to consider $x$ and $x \times y = xy$.  Since
$x$ is null, the alternative law gives $x(xy) = x^2 y = 0$ as desired.
So, the dot and cross product of adjacent outer vertices both vanish.
In other words, by Theorem \ref{thm:incidence_real}, they give points
in $\PC$ that are one roll away or less.

Thus, by the same theorem, outer vertices that are not opposite give
points in $\PC$ that are two rolls away or less.  It follows from
this theorem that such vertices are orthogonal.  

It remains to compute the dot product of opposite outer vertices.  By
the definition of null triple, the opposite vertices $z$ and $x
\times y$ have dot product $\frac{1}{2}$.  Let us check that this
forces other opposite vertices to also pair to $\frac{1}{2}$, for
instance $x$ and $y \times z$:
\begin{eqnarray*}
	-2(x \cdot (y \times z) ) & = & x(y \times z) + (y \times z) x \\
	& = & x(y z) + (y z) x \\
	& = & x(y z) -  (z y) x \\
	& = & (xy) z -  z (y x) - [ x, y, z ] - [ z, y, x ]  \\
	& = & (x \times y) z + z(x \times y) \\
	& = & -2(z \cdot (x \times y)) \\
	& = & -1 .
\end{eqnarray*}
In the fifth line, we use the fact that the associator, $[x,y,z] =
(xy)z - x(yz)$, is antisymmetric in its three arguments, thanks to
alternativity \cite{Schafer}. A very similar calculation shows that the third
opposite pair, $y$ and $z \times x$, also have dot product $\frac{1}{2}$.

Next we turn to the cross product of outer vertices.  We have already
seen that adjacent outer vertices have vanishing cross product.  For
vertices that are neither adjacent nor opposite, we need to show
their cross product gives the outer vertex between them if they are
multiplied in the order specified by the orientation of the arrows.
This is true by definition in three cases.  The other three cases
require a calculation.  For example, consider the cross product of
$y \times z$ and $z \times x$:
\begin{eqnarray*}
	(y \times z) \times (z \times x) & = & (yz)(zx) \\
	& = & (zy)(xz) \\
	& = & z(yx)z \\
	& = & z(y \times x)z \\
	& = & -z^2 (x \times y) + z(z(y \times x) + (y \times x)z) \\
	& = & - 2z (z \cdot (y \times x)) \\
	& = & z 
\end{eqnarray*}
where in the last step we use $(x \times y) \cdot z = \frac{1}{2}$ and
in the third step, we use a \define{Moufang identity}:
\[ (zy)(xz) = z(yx)z . \]
which holds in any alternative algebra \cite{Schafer}.  We can omit
some parentheses here thanks to alternativity.  Very similar
calculations apply for other pairs of vertices that are neither
opposite nor adjacent.

The cross product of opposite vertices equals half the vertex in the middle,
if we multiply them in the correct order.   That is, we claim:
\[ (x \times y) \times z = (y \times z) \times x = (z \times x) \times y . \]
In fact, this follows from the following identity:
\[ (u \times v ) \times u = 0 \] 
when $u$ and $v$ are null and orthogonal. Before verifying this identity, we
show how it implies the claim. Note that $x+z$ is null and orthogonal to the
null vector $y$. Thus, by the identity:
\[ ((x + z) \times y) \times (x + z) = 0 . \]
Using bilinearity to expand this expression, we get:
\[ (x \times y) \times x + (x \times y) \times z + (z \times y) \times x + (z \times y) \times z = 0 . \]
The first and last terms vanish by the identity, implying:
\[ (x \times y) \times z = (y \times z) \times x \]
as desired. A similar calculation shows this equals $(z \times x) \times y$. 

Let us verify that $(u \times v) \times u = 0$ when $u$ and $v$ are null and
orthogonal. Since $u$ and $v$ are orthogonal, they anticommute. Thus $u \times
v = uv$.  By the definition of the cross product:
\[ (uv) \times u = \frac{1}{2}((uv)u - u(uv)) = -u^2 v = (u \cdot u) v = 0 . \]
where we have made use of the anticommutativity of $u$ and $v$, along
with the fact that the split octonions are alternative, so that the
subalgebra generated by any two elements is associative \cite{Schafer}.

Next we compute the dot and cross product of each outer vertex with the middle
vertex, which we shall call $w$:
\[     w = 2 (x \times y) \times z. \]
To do this, we claim that these vectors:
\[ 1, \quad i = z + (x \times y) , \quad j = z - (x \times y), \quad k = i \times j = w \]
span a copy of the split quaternions in $\O'$. Before we verify this claim, let
us show how it determines the dot and cross product of the outer vertices with
$w$. In the split quaternions, $k$ is orthogonal to both $i$ and $j$, which
happens if and only if:
\[ w \cdot z = 0 , \quad w \cdot (x \times y) = 0. \]
Moreover, $k \times i = j$, which happens if and only if:
\[ w \times z = z, \quad w \times (x \times y) = -(x \times y), \]
The other cases work the same way, since we have shown that $(x \times y)
\times z$ is unchanged by cyclic permutations of the factors $x$, $y$ and $z$. 

Now let us check the claim that the  $1$, $i$, $j$ and $k$ as defined above
really do span a copy of the split quaternions. We need only check that $i^2 =
-j^2 = -1$, and that $i$ and $j$ anticommute, since it then follows that $k = i
\times j = ij$ anticommutes with $i$ and $j$ and squares to 1. For $i$, we have:
\begin{eqnarray*}
	i^2 & = & (z + x \times y)^2 \\
	& = & z^2 + z(x \times y) + (x \times y)z + (x \times y)^2 \\
	& = & -2z \cdot (x \times y) \\
	& = & -1
\end{eqnarray*}
where have used the fact that $z$ and $x \times y$ are null and pair to
$\frac{1}{2}$. A similar calculation shows $j^2 = 1$, and a further quick
calculation shows that $i$ and $j$ are orthogonal, and hence anticommute as
desired. As a bonus, we obtain the dot product of the middle vertex
$w$ with itself, because:
\[ w \cdot w = -w^2 = -k^2 = -1 . \]

Finally, let us verify that the seven vertices give a basis of $\I$.
Because $x$ and $y$ both give points of $\PC$ one roll away from the
point corresponding to $x \times y$, it follows from Proposition
\ref{lem:annihilator} that
\[ \Ann_{x \times y} = \langle x, x \times y, y \rangle . \]
It follows that these three vectors are linearly independent, since
they span a 3d null subspace.  Similarly:
\[ \Ann_z = \langle z \times x, z, y \times z \rangle . \]
Moreover, these annihilators must be complementary: any nonzero
element of $\Ann_{x \times y} \cap \Ann_z$ gives a point at most roll
away from $\langle x \times y \rangle$ and $\langle z \rangle$,
contradicting the fact that by Theorem \ref{thm:incidence_real} 
these points are more than two rolls apart, since the cross product 
$x \times y$ and $z$ is nonzero.  Thus the
vector space spanned by these two annihilators must be 6-dimensional,
and it remains to find a seventh vector independent of the six basis
vectors already named.  Since $w = 2 (x \times y) \times z$ is orthogonal to
both $\Ann_{x \times y}$ and $\Ann_z$, it is the seventh independent vector.
\end{proof}

After the hard work of proving the previous theorem, the next is
a direct consequence:

\begin{thm}
\label{thm:torsor}
The set of null triples is a $\G_2'$-torsor: given two null triples
$(x,y,z)$ and $(x',y',z')$, there exists a unique element of $g \in
\G_2'$ taking one to the other:
\[ (gx, gy, gz) = (x',y',z') . \]
\end{thm}
\begin{proof}
Because the action of $\G_2'$ preserves the dot and cross products, it
takes null triples to null triples. Moreover, the action of $g \in
\G_2'$ on $\I$ is determined by its action on a null triple, since a
null triple generates $\I$. Thus, there is at most one element of
$\G_2'$ taking $(x,y,z)$ to $(x',y',z')$. To see there is at least one
such element, consider the linear map:
\[ g \maps \I \to \I \]	
which maps the basis obtained from $(x,y,z)$:
\[ x, \quad y, \quad z, \quad x \times y, \quad y \times z, \quad z \times x, \quad 2(x \times y) \times z \]
to that obtained from $(x',y',z')$:
\[ x', \quad y', \quad z', \quad x' \times y', \quad y' \times z', \quad z' \times x', \quad 2(x' \times y') \times z' . \]
By Theorem \ref{thm:nulltriple}, this linear isomorphism preserves
the dot and cross product. Thus $g \in \G_2'$, as desired.
\end{proof}

\begin{prop}
\label{prop:pairtotriple}
Given any pair of null vectors $x, y \in \I$ such that $\langle x \rangle$
and $\langle y \rangle$ are two rolls away, there is a null vector $z \in
\I$ such that $(x,y,z)$ is a null triple.
\end{prop}

\begin{proof}
Recall that the vectors $x$, $y$ and $x \times y$ span a maximal null subspace:
\[ V = \Ann_{x \times y} = \langle x, x \times y, y \rangle . \]
Pick any maximal null subspace $W$ complementary to $V$.  The quadratic
form $Q$ must be nondegenerate when restricted to the direct sum
$V \oplus W \subset \I$.  Thus the map taking the dot product with $x$:
\[ \begin{array}{rcl} 
		W & \to & \R \\
		w & \mapsto & w \cdot x \\
\end{array}
\]
must have a two-dimensional kernel. Similarly, the map taking the dot product with $y$:
\[ \begin{array}{rcl} 
		W & \to & \R \\
		w & \mapsto & w \cdot y \\
\end{array}
\]
must also have a two-dimensional kernel. These kernels must be distinct, or
else $x$ and $y$ are proportional, contradicting their linear independence.
Thus, the subspace of $W$ orthogonal to both $x$ and $y$, the intersection of
these two-dimensional kernels in the three-dimensional $W$, is one-dimensional.
Let $z \in W$ span this intersection. By choice, $z$ is othogonal to $x$
and $y$, so it must have nonzero dot product with $x \times y$, because
otherwise the pairing on $V \oplus W$ would be degenerate. Thus:
\[ (x \times y) \cdot z \neq 0 . \]
By rescaling $z$ if necessary, we obtain:
\[ (x \times y) \cdot z = \frac{1}{2} . \qedhere \]
\end{proof}

We can use the preceding proposition to create a null triple starting from
any pair of null vectors, not just those whose projectivizations are two
rolls away.

\begin{prop} We have:
\label{prop:extensiontonulltriple}
\begin{enumerate}
	\setcounter{enumi}{-1}
	\item Any null vector $x \in \I$ is the first vector of some null
		triple $(x, y, z)$.
	\item Given any pair of null vectors $w, x \in \I$ such that
		$\langle w \rangle$ and $\langle x \rangle$ are one roll away,
		there is a null triple $(x,y,z)$ such that $w = x \times y$.
	\item Given any pair of null vectors $x, y \in \I$ such that
		$\langle x \rangle$ and $\langle y \rangle$ are two rolls away,
		there is a null vector $z \in \I$ such that $(x,y,z)$ is a
		null triple.
	\item Given any pair of null vectors $w, x \in \I$ such that
		$\langle w \rangle$ and $\langle x \rangle$ are three rolls
		away, there is a null triple $(x, y, z)$ such that $\langle w
		\rangle = \langle y \times z \rangle$.
\end{enumerate}
\end{prop}

\begin{proof}
We apply Proposition \ref{prop:pairtotriple} repeatedly. To prove part
0, choose any $y$ two rolls away from $x$. By Proposition
\ref{prop:pairtotriple}, there exists a $z$ such that $(x,y,z)$ is a
null triple. For part 1, choose $\langle y \rangle$ one roll away from $\langle
w \rangle$ not lying on the line joining $\langle w \rangle$ and $\langle x
\rangle$. By choice, $\langle x \rangle$ and $\langle y \rangle$ are two rolls
away, and Proposition \ref{prop:rolls} tells us that $\langle w \rangle$ is the
unique point at most one roll away from $\langle x \rangle$ and 
$\langle y \rangle$.  On the other hand, $\langle x \times y \rangle$ is 
also at most one roll from both $\langle x \rangle$ and $\langle y \rangle$. 
To see this, note from Theorem \ref{thm:incidence_real} that the $x$ and $y$ 
are orthogonal, and thus anticommute. Hence $x \times y = xy = -yx$, and 
a quick calculation shows that $x \times y$ annihilates both $x$ and $y$ 
because $x$ and $y$ are null.  By another application of Theorem 
\ref{thm:incidence_real}, we conclude $\langle x \times y \rangle$ is at most one roll 
from both $\langle x \rangle$ and $\langle y \rangle$.  From the uniqueness of $\langle w \rangle$, it 
follows that $\langle w \rangle = \langle x \times y \rangle$.  Rescaling 
$y$ if necessary, we have that $w = x \times y$.  Since $x$ and $y$ are two 
rolls apart, Proposition \ref{prop:pairtotriple} gives the result.  

Part 2 is just a restatement of Proposition \ref{prop:pairtotriple}.
Finally, for part 3, note that Theorem \ref{thm:incidence_real} implies $w
\cdot x \neq 0$.  Hence the linear map
\[ \begin{array}{ccc} 
	\Ann_w & \to & \R \\
	u & \mapsto & u \cdot x \\
\end{array}
\]
has rank one, and a two-dimensional kernel. Let $y$ and $z$ be
orthogonal vectors spanning this kernel. As in the proof of part 1,
$\langle y \rangle$ and $\langle z \rangle$ are each one roll away
from $\langle w \rangle$, but two rolls away from each other: if they
were one roll apart, $yz = 0$, and $\langle w, y, z \rangle$ would be
a three-dimensional null subalgebra. Since the maximal dimension of a
null subalgebra is two, we must have $yz \neq 0$. It now follows from
the argument in part 1 that $\langle w \rangle = \langle y \times z
\rangle$. Further, $(x,y,z)$ are pairwise orthogonal by construction.
We claim that $(x \times y) \cdot z \neq 0$, and so rescaling $z$ if
necessary, $(x,y,z)$ is a null triple.

To check this last claim, note that $\Ann_w = \langle w, y, z
\rangle$.  Moreover, $\Ann_x$ and $\Ann_w$ are complementary null
subspaces, both of maximal dimension: any nonzero vector in their
intersection would give a point that is one roll away from both
$\langle w \rangle$ and $\langle x \rangle$, contradicting our
assumption that these points are three rolls away.  The inner product
restricts to a nondegenerate inner product on the direct sum $\Ann_w
\oplus \Ann_x$.  In particular, since $x \times y \in \Ann_x$ is
orthogonal to itself and all vectors in $\Ann_x$, it must have a
nonvanishing inner product with some vector in $\Ann_w$, or else the
inner product would be degenerate. But $x \times y$ is also orthogonal
to $w$ and $y$, thanks to Theorem \ref{thm:incidence_real}: $\langle x
\times y \rangle$ is one roll away from $\langle y \rangle$, which is
one roll away from $\langle w \rangle$, so $\langle x \times y
\rangle$ and $\langle w \rangle$ are at most two rolls away. Thus, we
must have $(x \times y) \cdot z \neq 0$, as desired.
\end{proof}

We can use null triples to decompose $\PC \times \PC$ into its orbits under
$\G_2'$. There are precisely four:
\begin{thm}
\label{thm:orbits}
Under the action of $\G_2'$, the space of pairs of configurations, $\PC \times
\PC$, decomposes into the following orbits:
\begin{enumerate}
\setcounter{enumi}{-1}
\item $X_0 \subset \PC \times \PC$, the space of pairs zero
rolls away from each other. This is the diagonal set:
\[ X_0 = \{ (\langle x \rangle, \langle x \rangle ) \in \PC \times \PC \} .
 \]
\item $X_1 \subset \PC \times \PC$, the space of pairs one roll
away from each other:
\[ X_1 = \{ (\langle x \rangle, \langle y \rangle ) \in \PC \times \PC : 
\; \langle x \rangle \neq \langle y \rangle, \quad x \times y = 0 \} . \]
\item $X_2 \subset \PC \times \PC$, the space of pairs two
rolls away from each other:
\[ X_2 = \{ (\langle x \rangle, \langle y \rangle ) \in \PC \times \PC : 
\; x \times y \neq 0, \quad x \cdot y = 0 \} . \]
\item $X_3 \subset \PC \times \PC$, the space of pairs three
rolls away from each other:
\[ X_3 = \{ (\langle x \rangle, \langle y \rangle ) \in \PC \times \PC : 
\; x \cdot y \neq 0 \} . \]
\end{enumerate}
\end{thm}

\begin{proof}
In essence, we combine Theorem \ref{thm:torsor} with Proposition
\ref{prop:extensiontonulltriple}. 

To prove part 0, let $x$ and $x'$ be two nonzero null vectors in $C$. We claim
there is an element $g \in \G_2'$ such that $x' = gx$. Indeed, by Proposition
\ref{prop:extensiontonulltriple}, there are null triples $(x,y,z)$ and
$(x',y',z')$, so let $g$ be the element of $\G_2'$ guaranteed by Theorem
\ref{thm:torsor} taking the first null triple to the second. It follows that
$\G_2'$ acts transitively on nonzero null vectors, on $\PC$, and thus on the
diagonal subset $X_0$ of $\PC \times \PC$.

For part 1, let $\langle w \rangle$ and $\langle x \rangle$ be points
of $\PC$ that are one roll away.  By Proposition
\ref{prop:extensiontonulltriple}, there is a null triple $(x,y,z)$
such that $w = x \times y$. If $\langle w' \rangle$ and $\langle x'
\rangle$ are another pair of points that are one roll away, let
$(x',y',z')$ be a null triple such that $w' = x' \times y'$. Now let
$g \in \G_2'$ carry one null triple to the other. We then have $x' =
gx$ and
\[ w' = x' \times y' = gx \times gy = g(x \times y) = gw . \]
It follows that $\G_2'$ acts transitively on $X_1$. 

For part 2, let $\langle x \rangle$ and $\langle y \rangle$ be two
rolls away.  By Proposition \ref{prop:extensiontonulltriple}, there is
a null triple $(x,y,z)$. If $\langle x' \rangle$ and $\langle y'
\rangle$ is another pair of points that are two rolls away, there is a
null triple $(x',y',z')$. Letting $g \in \G_2'$ take one null triple to
the other, we immediately conclude that $\G_2'$ is transitive on $X_2$.

Finally, for part 3, let $\langle w \rangle$ and $\langle x \rangle$
be three rolls away. By Proposition \ref{prop:extensiontonulltriple},
there is a null triple $(x,y,z)$ such that $\langle w \rangle =
\langle y \times z \rangle$. If $\langle w' \rangle$ and $\langle x'
\rangle$ is another pair of points three rolls away, and $(x',y',z')$
a null triple such that $\langle w' \rangle = \langle y' \times z'
\rangle$, let $g \in \G_2'$ take one null triple to the other. Then $x'
= gx$, and
\[ \langle w' \rangle = \langle gy \times gz \rangle 
= \langle g(y \times z) \rangle = \langle gw \rangle . \]
It follows that $\G_2'$ acts transitively on $X_3$. 
\end{proof}

\section{Geometric quantization}
\label{quantization}

Recall that $C$ is the lightcone in the imaginary split octonions:
\[     C = \{ x \in \I : \; Q(x) = 0 \} .\]
Let $\PC$ be the corresponding real projective variety in 
$\Proj\I$:
\[    \PC = \{x \in C : x \ne 0\}/\R^*  .\]
This is called a `real projective quadric'.  We write $\langle x
\rangle$ for the 1-dimensional subspace of $\I$ containing the nonzero
vector $x \in \I$.  Then each point of $\PC$ can be written as
$\langle x \rangle$ for some nonzero $x \in C$.

The projectivized lightcone comes equipped with a real line bundle $L
\to \PC$ whose fiber over the point $\langle x \rangle$ consists of
linear functionals on $\langle x \rangle$:
\[     L_{\langle x \rangle} = \{ f \maps \langle x \rangle \to \R : \; 
f \textrm{\; is linear} \}. \]
In other words, $L$ is the restriction to $\PC$ of the dual of the
canonical line bundle on the projective space $\Proj\I$.

We would like to recover $\I$ from this line bundle over the
projectivized lightcone via some process of `geometric quantization'.
However, this process is best understood for holomorphic line bundles
over K\"ahler manifolds \cite{Hurt,Woodhouse}.  So, we start by 
complexifying everything.

If we complexify the split octonions we obtain an algebra $\C \otimes \O'$
over the complex numbers which is canonically isomorphic to the
complexification of the octonions, $\C \otimes \O$.  This latter
algebra is called the `bioctonions'.  So, we call the complexification
of the split octonions the {\bf bioctonions}, and denote it simply by
$\O^\C$.  The quadratic form $Q$ on $\O'$ extends to a complex-valued
quadratic form on $\O^\C$, which we also denote as $Q$.  This
quadratic form makes $\O^\C$ into a composition algebra.  We can
polarize $Q$ to obtain the \define{dot product} on $\O^\C$, the unique
symmetric bilinear form for which
\[   x \cdot x = Q(x)   \]
for all $x \in \O^\C$.

Complexifying the subspace $\I \subset \O'$ gives a 
7-dimensional complex subspace of $\O^\C$.  We denote 
this as $\I^\C$ and call its elements \define{imaginary bioctonions}.
We define
\[    C^\C = \{ x \in \I^\C : \; Q(x) = 0 \}. \]
This is what algebraic geometers would call a `complex quadric'.
We define $\PC^{\C}$ to be the corresponding projective variety in
the complex projective space $\Proj\I^\C$:
\[    \PC^\C = \{x \in C^\C : x \ne 0\}/\C^*  .\]
This is a `complex projective quadric'.  

If we now change notation slightly and write $\langle x \rangle$ for
the 1-dimensional \emph{complex} subspace of $\I^\C$ containing the
nonzero vector $x \in \I^\C$, then each point of $\PC^\C$ is of the form
$\langle x \rangle$ for some nonzero $x \in C^\C$.  The complex
projective quadric $\PC^\C$ comes equipped with a holomorphic complex
line bundle $L^\C \to \PC^\C$ whose fiber at $\langle x \rangle$
consists of all complex linear functionals on $\langle x \rangle$:
\[     L_{\langle x \rangle} = \{ f \maps \langle x \rangle \to \C : \; 
f \textrm{\; is linear} \} .\]

Since the real projective quadric $\PC$ is included in the complex
projective quadric $\PC^\C$, and similarly the total space of
the line bundle $L$ is included in the total space of the complex
line bundle $L^\C$, we have a commutative diagram:
\[\xymatrix{
 L \ar[d] \ar[r] & L^\C \ar[d] \\
 \PC \ar[r] & \PC^\C 
}\]
Complex conjugation gives rise to a conjugate-linear map from
$\I^\C = \C \otimes \I$ to itself, whose set of fixed points is
just $\I$.  This in turn gives an antiholomorphic map from $\PC^\C$
to itself whose fixed points are just $\PC$.  This lifts to an 
antiholomorphic map from $L^\C$ to itself whose fixed points are just
$L$.  So, we actually have a commutative diagram 
\[\xymatrix{
 L \ar[d] \ar[r] & L^\C \ar[d] \ar@(ur,dr)[] \\
 \PC \ar[r] & \PC^\C  \ar@(ur,dr)[]
}\]

Now we shall show that every imaginary bioctonion gives a holomorphic
section of $L^\C$, and that every holomorphic section of $L^\C$
arises this way.  Even better, every imaginary split octonion gives a
section of $L$ that extends to a holomorphic section of $L^\C$, 
and every section of $L$ with this property arises that way.

First, note that every imaginary bioctonion $w$ gives a
section $s_w$ of $L^\C$ as follows.  Evaluated at point $\langle x
\rangle$ of $\PC^\C$, $s_w$ must be some linear functional on the span of
$x \in C$.  We define this linear functional so that it maps $x$ to
$w \cdot x$:
\[      s_w(x) = w \cdot x  .\]
It is easy to check that the section $s_w$ is holomorphic.  Moreover,
every holomorphic section of $L^\C$ arises this way:

\begin{thm} 
\label{thm:complex_sections}
A section of $L^\C$ is holomorphic if and only if it is of the form
$s_w$ for some imaginary bioctonion $w$, which is then unique.  Thus,
the space $\I^\C$ of imaginary bioctonions is isomorphic to the space
of holomorphic sections of $L^\C$ over $\PC^\C$.  
\end{thm}

\begin{proof}  
This is a direct consequence of the Bott--Borel--Weil Theorem
\cite{Bott,Serre}.  Any finite-dimensional irreducible complex
representation of a complex semisimple Lie group $G$ arises as the
space of holomorphic sections of a holomorphic line bundle over $G/P$
for some parabolic subgroup $P$.  In particular, the irreducible
representation of $\G^\C_2$ on $\I^\C$ arises as the space of holomorphic
sections of $L^\C \to \PC^\C$, where $\PC^\C \cong \G^\C_2/P$ with $P$
being the subgroup that fixes a 1d null subspace in $\I^\C$.  This
implies that sections of the form $s_w$ are all the holomorphic
sections of $L^\C$.  Clearly different choices of $w$ give different
sections $s_w$.
\end{proof}

We can think of $\I$ as a real subspace of $\I^\C$, and
then each $w \in \I$ gives a section $s_w$ of $L^\C$ using the
same construction.  However, since $w \cdot x$ is real
when $w,x \in \I$, restricting this section to $\PC \subset \PC^\C$
actually gives a section of the real line bundle $L$.  Moreover:

\begin{thm} 
\label{thm:real_sections}
A section of $L \to \PC$ extends to a global holomorphic section of
$L^\C \to \PC^\C$ if and only if it is of the form $s_w$ for some
imaginary split octonion $w$, which is then unique.  Thus, the space
$\I$ of imaginary split octonions is isomorphic to the space of
sections of $L$ over $\PC$ that extend to holomorphic sections of $L^\C$
over all of $\PC^\C$.
\end{thm}

\begin{proof} 
Suppose we have a section of $L$ that extends to a global holomorphic
section of $L^\C$.  By Theorem \ref{thm:complex_sections}, this 
holomorphic section of $L^\C$ is of the form $s_w$ for some imaginary
bioctonion $w$.  Its restriction to $\PC \subset \PC^\C$ will lie
in $L$ if and only if $w \cdot x$ is real for all $x \in C$.
This is true if and only if $w \in \I$.  Clearly different choices
of $w$ give different sections $s_w$.
\end{proof}

\section{The cross product from quantization}
\label{cross product}

It would be nice if we could use geometric quantization to recover not
only the vector space of imaginary octonions, but also the octonions
together with their algebra structure.  Here we do this for the
bioctonions.  We already know from Theorem \ref{thm:complex_sections}
that geometric quantization of the space $\PC^\C$ gives a vector
space isomorphic to the imaginary bioctonions.  Now we will use
geometric quantization to equip this space with an operation that
matches the \define{cross product} of imaginary bioctonions:
\[        x \times y = \frac{1}{2} \left(x y - y x\right). \]
This, we claim, is enough to recover the bioctonions as an algebra.

To see this, note that we can write $\O^\C = \C \oplus \I^\C$
where $\C$ consists of complex multiples of the identity $1 \in \O^\C$.
To describe the multiplication of bioctonions it is thus enough to say
what happens when we multiply two imaginary bioctonions.  But the product 
of two imaginary bioctonions obeys
\[         x y = x \times y - x \cdot y  \]
where $x \times y$ is an imaginary bioctonion and $x \cdot y$ is a
multiple of the identity.  Here the \define{dot product} arises from
polarizing the quadratic form on the bioctonions:
\[       x \cdot x = Q(x) , \]
but on imaginary bioctonions it is also proportional to the anticommutator:
\[        x \cdot y = - \frac{1}{2} \left( x y + y x \right) .\]
All this is easy to check by explicit computation.

Thus, to describe the bioctonions as an algebra it is enough to
describe the cross product and dot product of imaginary bioctonions.
Explicitly, multiplication in $\O^\C = \C \oplus \I^\C$ is given by
\[     (\alpha, a) (\beta, b) 
= (\alpha \beta - a \cdot b, \, \alpha b + \beta a + a \times b) . \]
But in fact, the dot product can be recovered from the cross product:
\[     a \cdot b = - \frac{1}{6} \tr(a \times (b \times \cdot)) \]
where the right-hand side refers to the trace of the map
\[      a \times (b \times \cdot) \maps \I^\C \to \I^\C  .\]
It is clear that some such formula should be true, since $\I^\C$ is an
irreducible representation of $\G_2^\C$, the complex form of $\G_2$,
so any two invariant bilinear forms are proportional.  The constant
factor can thus be checked by computing the trace of the operator $(a
\times (a \times \cdot))$ for a single imaginary bioctonion $a$;
briefly, we get $-6 a \cdot a$ because there is 6-dimensional subspace
orthogonal to $a$, on which this operator acts as multiplication by
$-a \cdot a$.

In short, the whole algebra structure of the bioctonions can be
recovered from the cross product of imaginary bioctonions.  We
can even define $\G_2^\C$ to be the group of linear transformations 
of the imaginary bioctonions that preserve the cross product.

Thus it is interesting to see if we can construct the cross product using
geometric quantization.  In fact we can.  The procedure uses a `correspondence'
between the complex manifolds $\PC^\C$ and $\PC^\C \times \PC^\C$, which is
a diagram like this:
\[\xy
	(0,18)*{}="A";
	(-16,0)*{}="B";
	(16,0)*{}="C";
	{\ar_{p} (-2,15);(-14,3)};
	{\ar^{i} (2,15);(14,3)};
	"A"*{S};
	"B"*{\PC^\C};  
	"C"*{\PC^\C \times \PC^\C};
\endxy
\]
where the maps $p$ and $i$ exhibit $S$ as a complex submanifold
embedded in $(\PC^\C)^3$.  But in the case we shall consider, $i$ by itself is
already an embedding.

We can use $p$ to pull the line bundle $L^\C$ from $\PC^\C$ back to
$S$.  Then, since $i$ is an embedding, we can push the resulting line
bundle forward to the submanifold $i(S) \subset \PC^\C \times \PC^\C$.  But
there is another line bundle on $\PC^\C \times \PC^\C$: the \define{external
tensor product} of the dual canonical bundle with itself, $L^\C \boxtimes
L^\C$, whose fiber over any point $(a,b) \in \PC^\C \times \PC^\C$ is $L_a
\tensor L_b$.  The bundle $L^\C \boxtimes L^\C$ restricts to a bundle over
$i(S)$.  This is potentially different than the bundle obtained by pulling back
$L^\C$ along $p$ and then pushing it forwards along $i$.  In Proposition
\ref{prop:iso}, however, we show these line bundles on $i(S)$ can be
identified.

Recall that imaginary bioctonions can be identified with sections of
$L^\C$.  By the construction described so far, we can take any such
section, pull it back to $S$, push it forward to $i(S)$, and then
think of it as a section of $L^\C \boxtimes L^\C$ restricted to
$i(S)$.  In Proposition \ref{prop:unique_extension} we show this
section extends to all of $\PC^\C \times \PC^\C$.  The result can be
identified with an element in the tensor square of the space of
imaginary bioctonions.  All in all, this procedure gives rise to a
linear map
\[          \Delta \maps \I^\C \to \I^\C \otimes \I^\C . \]
This is a kind of `comultiplication' of imaginary bioctonions.
However, the space of imaginary bioctonions can be identified with its
dual using the dot product.  This gives a linear map
\[          \Delta^* \maps  \I^\C \otimes \I^\C \to \I^\C  .\]
and in Theorem \ref{thm:cross_product} we show that this is the cross
product, at least up to a nonzero constant factor.

The most interesting fact about this whole procedure is that the
correspondence
\[\xy
	(0,18)*{}="A";
	(-16,0)*{}="B";
	(16,0)*{}="C";
	{\ar_{p} (-2,15);(-14,3)};
	{\ar^{i} (2,15);(14,3)};
	"A"*{S};
	"B"*{\PC^\C};  
	"C"*{\PC^\C \times \PC^\C};
\endxy
\]
can be defined \emph{using solely the incidence geometry of} the
projective lightcone $\PC^\C$: in other words, using only points
and lines in this space, and the relation of a point lying on a line.

To do this, we start by extending the concept of line from the $\PC$
to its complexification $\PC^\C$.  Following Theorem
\ref{thm:lines_and_2d_null_subalgebras}, we define a \define{line} in
$\PC^\C$ to be the projectivization of a 2d null subalgebra of the
bioctonions.  This is also the concept of line implicit in the Dynkin
diagram of $\G_2$: in the theory of buildings, given any simple Lie
group, each dot in its Dynkin diagram corresponds to a type of figure
in a geometry having that group as symmetries \cite{Brown}.  The
details for $\G_2$ are nicely discussed by Agricola \cite{Agricola}.

Given this concept of line, we can describe the correspondence of
complex manifolds that yields a geometric description of the
bioctonion cross product.  We begin by defining the manifold $S$.  

First, recall from Definition \ref{defn:rolls} and Proposition
\ref{prop:rolls} that two points $a, b \in \PC^\C$ are `one roll away' if
$a \ne b$ but there is some line containing both $a$ and $b$.  We
define $S$ to be the subset of $(\PC^\C)^3$ consisting of triples
$(a,b,c)$ for which $b$ is the only point that is one roll away from
both $a$ and $c$.  In this situation $a$ and $c$ are `two rolls away',
and we call $b$ the \define{midpoint} of $a$ and $c$.  We shall soon
see that if $a = \langle x \rangle$ and $c = \langle z \rangle$ are
two rolls away, their midpoint is $b = \langle x \times z \rangle$.
So, the cross product is hidden in the incidence geometry, and we can
use geometric quantization to extract it.

Next we define $p$ and $i$.  The map $p$ picks out the midpoint:
\[
\begin{array}{cccl}
p \maps & S & \to & \PC^\C \\
        & (a,b,c) & \mapsto & b .
\end{array}
\]
The map $i$ picks out the other two points:
\[
\begin{array}{cccl}
i \maps & S & \to & (\PC^\C)^2 \\
        & (a,b,c) & \mapsto & (a,c).
\end{array}
\]

Next we show that $S$ is a complex manifold and $i$ is an embedding.  We
also show that $p$ makes $S$ into the total space of a fiber bundle over
$\PC^\C$, though we will not need this fact.  To get started, we must
relate the geometry of $\PC^\C$ to operations on the space of
imaginary split octonions:

\begin{thm} 
\label{thm:incidence_complex}
Suppose that $\langle x \rangle, \langle y \rangle
\in \PC^\C$.  Then:

\begin{enumerate}
\item $\langle x \rangle$ and $\langle y \rangle$ are at most one roll
away if and only if $x y = 0$, or equivalently, $x \times y = 0 $.
\item $\langle x \rangle$ and $\langle y \rangle$ are at most two rolls
away if and only if $x \cdot y$ = 0.
\item $\langle x \rangle$ and $\langle y \rangle$ are always at most
three rolls away.
\end{enumerate}
\end{thm}

\begin{proof}
The proof here is exactly like that of Theorem \ref{thm:incidence_real}, 
so we omit it.  In particular, like the split octonions, the bioctonions 
are an alternative algebra \cite{Schafer}.
\end{proof}

We define the \define{annihilator} of a nonzero element $x \in C^\C$ to
be this subspace of $\I^\C$:
\[       \Ann_x = \{ y \in \I^\C : \; xy = 0 \}  .\]

\begin{prop} 
\label{one_roll_or_less}
Given a point $\langle x \rangle \in \PC^\C$, the set of points that
are at most one roll away from $\langle x \rangle$ is the projectivization 
of $\Ann_x$.
\end{prop}

\begin{proof}
Theorem \ref{thm:incidence_complex} says that $\langle y \rangle \in
\PC^\C$ is at most one roll away from $\langle x \rangle$ if and only
if $xy = 0$.
\end{proof}

\begin{prop}
\label{prop:annihilator}
Suppose $y \in C^\C$ is nonzero.  Then 
$\Ann_y$ is a 3-dimensional null subspace of $\I^\C$, and any two
elements of $\Ann_y$ anticommute.
\end{prop}

\begin{proof}
Consider two nonzero elements $x,z \in \Ann_y$. They anticommute if they
have vanishing dot product, since their anticommutator $xz + zx$ is
proportional to their dot product.  

So, we need only show that $\Ann_y$ is null and 3-dimensional.  Since
$\G_2^\C$ acts transitively on $\PC^\C$, it suffices to prove this
for a single chosen $y \in C^\C$.  We do the special case where $y$
actually lies in $\I \subset \I^\C$.  In this case, we know from
Lemma \ref{lem:annihilator} that $\{x \in \I \colon yx = 0\}$ is a
3-dimensional null real subspace of $\I$.  Since $\Ann_y$ is the
complexification of this space, it is a 3-dimensional null complex
subspace of $\I^\C$.
\end{proof}

Now we are ready to study the set $S$:

\begin{prop}
\label{prop:algebraic}
The set $S$ is given by
\[  S = \{(\langle x \rangle , \langle y \rangle , \langle z \rangle)  \in (\PC^\C)^3 \colon \; xy = 0 = yz , \; xz \ne 0 \} .\]
\end{prop}

\begin{proof}  
By Theorem \ref{thm:incidence_complex},
the conditions say that $\langle x \rangle$ is one roll away from
$\langle y \rangle$ and $\langle y \rangle$ is one roll away from
$\langle z \rangle$ but $\langle x \rangle$ is not one roll away
from $\langle z \rangle$.  This is a way of saying that $\langle x \rangle$ is 
two rolls from $\langle z \rangle$, with $\langle y \rangle$ as their 
midpoint, which is the condition for this triple of points to be in
$S$.
\end{proof}

It is also useful to express the set $S$ in terms of the cross
product, as promised above:

\begin{prop}
\label{prop:midpointcrossproduct}
Let $(\langle x \rangle, \langle y \rangle , \langle z \rangle)$ be a
point of $S$. Then $\langle y \rangle = \langle x \times z \rangle$.
\end{prop}

\begin{proof}  
By Proposition \ref{prop:algebraic}, we know that $xy = 0 = yz$ and
$xz \ne 0$.  Noting that $y = xz$ is a solution to the first two
equations because $x$ and $z$ are null, we must have $\langle y \rangle
= \langle xz \rangle$ because the midpoint is unique. So, it suffices
to check that $xz = x \times z$. Since the cross product is half the
commutator, this happens precisely when $x$ and $z$ anticommute. By
Proposition \ref{prop:annihilator}, $x$ and $z$ indeed anticommute,
because they both lie in the annihilator of $y$.
\end{proof}

\begin{prop}
\label{prop:altalg}
The set $S$ is given by
\[  S = \{(\langle x \rangle , \langle x \times z \rangle , \langle z \rangle)  \in (\PC^\C)^3 \colon \; x \cdot z = 0, \quad x \times z \neq 0 \} . \]
\end{prop}

\begin{proof}
By Theorem \ref{thm:incidence_complex} the conditions here say that
$\langle x \rangle$ is two rolls away from $\langle z \rangle$, and 
we know from Proposition \ref{prop:midpointcrossproduct} that in 
this case their midpoint is $\langle x \times z \rangle$. 
\end{proof}

\begin{prop} 
We have:
\begin{enumerate}  
	\item $i(S) = \{ (\langle x \rangle, \langle z \rangle ) : \;
	x \cdot z = 0, \quad x \times z \neq 0 \}$.
	\item $i(S)$ is a complex submanifold of $\PC^\C \times \PC^\C$.
	\item $S$ is a complex submanifold of $(\PC^\C)^3$.
	\item $i \maps S \to \PC^\C \times \PC^\C$ is an embedding
	of $S$ as a complex submanifold of $\PC^\C \times \PC^\C$.
	\end{enumerate}
\end{prop}

\begin{proof}

Part 1 is clear from Proposition \ref{prop:altalg}.

For Part 2, to show $i(S)$ is a complex submanifold of $\PC^\C$, we show that
its preimage under the quotient map
\[ q \maps (C^\C - 0)^2 \to \PC^\C \times \PC^\C  \]
is a submanifold. Let us call this preimage $X$:
\[ X = q^{-1}( i(S) ) . \]
The set $X$ is contained in the open set $U$ on which the cross product
is nonvanishing:
\[ X \subset U = \{ (x,z) \in (C^\C - 0)^2 : \; x \times z \neq 0 \} . \]
This open set is a complex submanifold itself. On this open set, we can verify
that the dot product is a map of constant rank:
\[ \begin{array}{cccc} f \maps & U & \to & \C \\
		& (x,z) & \mapsto & x \cdot z .
	\end{array}  
\]
It follows that the preimage of zero under this map of constant rank, $X =
f^{-1}(0)$, is a submanifold.

To check that $f$ indeed has constant rank, we compute the rank of its
derivative at the point $(x,z) \in U \subset (C^\C - 0)^2$:
\[ \begin{array}{cccc} f_* \maps & T_{(x,z)} U & \to & \C \\  \\
	& (\dot{x},\dot{z}) & \mapsto & \dot{x} \cdot z + x \cdot \dot{z} .
	\end{array} 
\]
The linear map $f_*$ has rank one if and only if it is nonzero. Here,
since $U$ is an open subset of $(C^\C - 0)^2$, $\dot{x}$ is a tangent
vector to $C^\C$ at the point $x$, which we can identify with the set of
all vectors in the ambient vector space, $\I^\C$, that are orthogonal
to $x$. Likewise, $\dot{z}$ is in the set of all vectors
orthogonal to $z$. The derivative $f_*$ vanishes if and only if these
sets are equal: if $x$ is a nonzero multiple of $z$. But then their
cross product will vanish, so such pairs are excluded from $U$. Thus
$f$ has constant rank on $U$.

Finally, because $X$ is the inverse image of some set under the
quotient map $q$, its image $i(S)$ is also a submanifold. This
completes the proof of part 2.

Part 3 follows because $S$ is the graph of a holomorphic function 
\[ \begin{array}{cccc} g \maps & i(S) & \to & \PC^\C \\
		 & (\langle x \rangle, \langle z \rangle) & \mapsto & \langle x \times z \rangle . \\
	 \end{array} 
\]
The graph of this function is $S$, and is a submanifold of $i(S) \times
\PC^\C$, and in turn this is a submanifold of $(\PC^\C)^3$.

For part 4, note that $i$ is an embedding of complex manifolds because
the map from the graph of a holomorphic map to its domain is always an
embedding.
\end{proof}

As a side-note, we have:

\begin{prop} $p \maps S \to \PC^\C$ is a holomorphic
fiber bundle.
\end{prop}

\begin{proof} 
Over any point $\langle y \rangle \in \PC^\C$ the fiber of $p$ is $\Proj
\Ann_y \times \Proj \Ann_y$ with those pairs of points that are not two
rolls apart removed.  For this we need to remove any pair that
includes $\langle y \rangle$, as well as any pair on the diagonal
subset, $D$.  Thus the fiber is the complex manifold
\[ p^{-1}\langle y \rangle = \Proj \Ann_y \times \Proj \Ann_y \; - \; 
\langle y \rangle \times \Proj \Ann_y \; - \; \Proj \Ann_y \times \langle y
\rangle \; - \; D. \]
We leave the proof of local triviality as an exercise for the reader,
since we will not be using this fact.
\end{proof}

Since $i$ is a complex analytic diffeomorphism onto its image $i(S)$,
we can push forward any holomorphic line bundle $\Lambda$ on $S$ to a
holomorphic line bundle over $i(S)$, which we call $i_* \Lambda$.
Thus, we obtain a holomorphic line bundle $i_* p^* L^\C$ over $i(S)$.
However, this is isomorphic to the line bundle $L^\C \boxtimes L^\C$
restricted to $i(S)$.

To see this, note that the fiber of $i_* p^* L$ over a point $(\langle
x \rangle, \langle z \rangle)$ of $i(S)$ is $L_{\langle x \times z
\rangle}$, the dual of the line $\langle x \times z \rangle$ in
$\I^\C$. On the other hand, the fiber $L^\C \boxtimes L^\C \big|_{i(S)}$ is
$L_{\langle x \rangle} \tensor L_{\langle z \rangle}$. Since the cross
product gives a map:
\[ \langle x \rangle \tensor \langle z \rangle \to 
\langle x \times z \rangle \]
dualizing yields a map:
\[ \Theta_{(\langle x \rangle, \langle z \rangle)} \maps 
L_{\langle x \times z \rangle} \to 
L_{\langle x \rangle} \tensor L_{\langle z \rangle} . \]
This may seem like a deceptive trick, since in a moment will use $\Theta$ to
construct the cross product. However, $\Theta$ can be characterized in other
ways, at least up to a constant multiple:

\begin{prop}
\label{prop:iso}
The map 
\[\Theta \maps i_* p^* L^\C \to L^\C \boxtimes L^\C \big|_{i(S)}  \]
is an isomorphism of holomorphic line bundles 
that is equivariant with respect to the action of $\G_2^\C$. Moreover, any other
$\G_2^\C$-equivariant map between these line bundles is a constant multiple of
$\Theta$. 
\end{prop}

\begin{proof} 
The map $\Theta$ is holomorphic by construction, and because the cross
product $x \times z$ is nonzero for $(\langle x \rangle, \langle z
\rangle) \in i(S)$, $\Theta$ is an isomorphism on each fiber.  Since
everything used to construct $\Theta$ is $\G_2^\C$-equivariant,
$\Theta$ is as well.

Now let $\Theta'$ be another $\G_2^\C$-equivariant map:
\[\Theta' \maps i_* p^* L^\C \to L^\C \boxtimes L^\C \big|_{i(S)}  \]
To prove the claim that $\Theta'$ is a constant multiple of $\Theta$, first
recall that Theorem \ref{thm:orbits} states that the set of pairs in $\PC \times
\PC$ that are two rolls apart is an orbit of $\G'_2$. This result is
straightforward to generalize to the complexification, and so we conclude that the
set of pairs in $\PC^\C \times \PC^\C$ that are two rolls apart is an orbit
of $\G_2^\C$. This is the set $i(S)$, so $\G_2^\C$ acts transitively on $i(S)$.

Picking any point $(a,c) \in i(S)$, we have 
$\Theta'_{(a,c)} = \alpha \Theta_{(a,c)}$
for some constant $\alpha$, since $\Theta_{(a,c)}$ spans the one-dimensional
space of maps between the one-dimensional fibers. Using the transitive action
of $\G_2^\C$ and the equivariance of $\Theta$ and $\Theta'$, we conclude
$\Theta' = \alpha \Theta$.
\end{proof}

In what follows, we use $\Gamma$ to denote the space of global
holomorphic sections of a holomorphic line bundle over a complex
manifold:

\begin{prop}
\label{prop:unique_extension}
There is a linear map 
\[  \Delta \maps \Gamma(L^\C) \to \Gamma(L^\C \boxtimes L^\C)  \]
that is equivariant with respect to the action of $G_2^\C$ and has the
property that if $\psi \in \Gamma(L^\C)$, then $\Delta \psi$ extends
$\Theta i_* p^* \psi$ from $i(S)$ to all of $\PC^\C \times \PC^\C$.
\end{prop}

\begin{proof} 
For any point $\langle y \rangle \in \PC^\C$, $\psi_{\langle y
\rangle}$ is an element of $L_{\langle y \rangle}$, meaning a linear
functional on the 1-dimensional subspace $\langle y \rangle$.  By
Theorem \ref{thm:complex_sections}, there exists an imaginary
bioctonion $w$ such that $\psi = s_w$.  In other words, $\psi$ is
determined by the fact that
\[       \psi_{\langle y \rangle} \maps  y \mapsto w \cdot y .\]
The fiber of $p^* L^\C$ over $(\langle x \rangle, \langle y \rangle,
\langle z \rangle)$ is just $L_{\langle y \rangle}$, and by definition
of the pullback, 
\[ (p^* \psi)_{(\langle x \rangle , \langle y \rangle, \langle z \rangle)} 
\maps  y \mapsto w \cdot y .\]
However, by Proposition \ref{prop:midpointcrossproduct}, we know
$\langle y \rangle = \langle x \times z \rangle$.  Pushing forward along $i$,
we thus have
\[ (i_* p^* \psi)_{(\langle x \rangle , \langle z \rangle)} 
\maps  x \times z \mapsto w \cdot (x \times z) .\]
Finally, applying $\Theta$, we obtain
\[ (\Theta i_* p^* \psi)_{(\langle x \rangle , \langle z \rangle)} 
\maps x \tensor z \mapsto w \cdot (x \times z) .\] 
Since this formula makes sense for all pairs $(\langle x \rangle,
\langle z \rangle)$, and not just those in $i(S)$, we see that $\Theta
i_* p^* \psi$ extends to a global section of $L^\C \boxtimes L^\C$
over $\PC^\C \times \PC^\C$, given by
\[ (\Delta \psi)_{(\langle x \rangle , \langle z \rangle)} 
\maps x \tensor z \mapsto w \cdot (x \times z) .\] 
This section $\Delta \psi$ is holomorphic because
$\psi$ and the cross product are both holomorphic.  Moreover,
$\Delta \psi$ depends linearly on $\psi$, and it is equivariant
by construction.
\end{proof}

Using the canonical isomorphism
\[  \Gamma(L^\C \boxtimes L^\C) \cong \Gamma(L^\C) \otimes \Gamma(L^\C) \] 
together with the isomorphism 
\[   \Gamma(L^\C) \cong \I^\C \]
given by Theorem \ref{thm:complex_sections}, we can reinterpret $\Delta$ 
as a linear map
\[     \Delta \maps \I^\C \to \I^\C \otimes \I^\C.  \]
Furthermore, $\I^\C$ is canonically identified with its dual using
the dot product of imaginary bioctonions.  This allows us to 
identify the adjoint of $\Delta$ with a linear map we call
\[   \Delta^* \maps \I^\C \otimes \I^\C \to \I^\C .  \]

\begin{prop}
\label{prop:cross_product}
The adjoint $\Delta^* \maps \I^\C \otimes \I^\C \to \I^\C $ is the
cross product.
\end{prop}

\begin{proof}
In the proof of Theorem \ref{prop:unique_extension} we saw that, up to a
nonzero constant factor, $\Delta$ sends the section $s_w$ of $L^\C$ to the
section of $L^\C \boxtimes L^\C$ given by
\[      (\Delta s_w)_{(\langle x \rangle , \langle z \rangle)}
\maps x \otimes z \mapsto w \cdot (x \times z)  .\]
This means that the adjoint of $\Delta$ is the cross product.  
\end{proof}

So far our construction may seem like `cheating', since we used the
cross product to define the map $\Delta$ whose adjoint is the cross
product.  However, we now show that that \emph{any} map with some of
the properties of $\Delta$ must give the cross product up to a
constant factor:

\begin{thm}
\label{thm:cross_product}
Suppose 
\[  \delta \maps \Gamma(L^\C) \to \Gamma(L^\C \boxtimes L^\C)  \]
is any linear map that is equivariant with respect to the action of
$\G_2^\C$.   Then identifying $\Gamma(L^\C)$ with
the imaginary bioctonions, the adjoint
\[    \delta^* \maps \I^\C \otimes \I^\C \to \I^\C \]
is the cross product up to a nonzero constant factor.
\end{thm}

\begin{proof}
By construction, $\delta^*$ is an intertwining operator between
representations of $\G_2^\C$.  However, any such intertwiner is a
constant multiple of the cross product, because the tensor square of
the 7-dimensional irreducible representation of $\G_2^\C$ contains
that irreducible representation with multiplicity one.
\end{proof}

\section{Conclusions}

The final theorem above raises a question.
What does our construction of the cross product of imaginary octonions
mean in terms of the physics of a rolling ball?  For a preliminary answer, 
we can naively imagine a section of $L^\C \to \PC^\C$ as a wavefunction
describing the quantum state of a rolling ball.  Then we can take such
a quantum state, `duplicate' it to get a quantum state of two new
rolling balls of which the original one was the midpoint, and extend
this to get a quantum state of an arbitrary pair of rolling balls.
This procedure gives a linear map
\[           \I^\C \to \I^C \otimes \I^C \]
whose adjoint is the cross product.  

However, this account is at best only roughly correct.  Since the real
projective quadric $\PC$ is the configuration space of a rolling
spinorial ball on a projective plane, we would expect to quantize this
system by forming the cotangent bundle $T^* \PC$, a symplectic
manifold, and applying some quantization procedure to that.  Instead
we passed to the complexification $\PC^\C$, which is in fact K\"ahler,
and applied geometric quantization to that.  Since there are
neighborhoods of $\PC$ in $\PC^\C$ that are isomorphic as
symplectic manifolds to neighborhoods of the zero section of $T^* \PC$, we can
loosely think of $\PC^\C$ as a way of modifying the
cotangent bundle to make it compact.  In a rough sense this amounts to
putting a `speed limit' on the motion of the rolling ball, making its
Hilbert space of states finite-dimensional.  However, it would be good
to understand this more precisely.  This might also clarify the
physical significance, if any, of the real vector space $\I$ obtained
by placing an extra condition on the vectors in the space $\I^\C$ obtained by
geometrically quantizing $\PC^\C$.

\section*{Acknowledgements}

This work began with a decade's worth of conversations with James Dolan, and
would not have been possible without him. We thank Jim Borger, Mike Eastwood,
Katja Sagerschnig, Ravi Shroff, Dennis The and Travis Willse for helpful
conversations, and Matthew Randall for catching some errors. We also thank the
referees for suggesting improvements.  This research was supported under the
Australian Research Council's Discovery Projects funding scheme (project number
DP110100072), and by the FQXi minigrant `The Octonions in Fundamental Physics'.
We thank the Centre for Quantum Technologies for their hospitality.

\end{document}